\documentclass[a4paper,10pt]{article}
\usepackage[utf8]{inputenc}
\usepackage{fullpage}
\usepackage{graphicx,xcolor}
\usepackage{amsmath,amsfonts,amsthm,amssymb}
\usepackage[final]{showkeys}
\usepackage{hyperref}
\usepackage{xfrac}
\usepackage[color=gray, backgroundcolor=gray!10, textwidth=2cm, textsize=footnotesize]{todonotes}
\usepackage{mathrsfs}
\usepackage{physics}
\usepackage{enumerate}
\usepackage{enumitem}

\newcommand\weakto{\rightharpoonup}
\newcommand\eps{\varepsilon}
\newcommand\R{\mathbb{R}}

\newcommand\N{\mathbb{N}}

\newcommand\calL{\mathcal{L}}

\newcommand\dt{{\mathrm d}t}

\usepackage{mathtools}                  
\usepackage{cleveref}
\newcommand{\less}{<}
\newcommand{\gtr}{>}
\newcommand{\hm}{\mathcal{H}}
\newcommand{\lm}{\mathcal{L}}
\newcommand{\diff}{\mathrm{D}}

\newcommand{\bv}{\mathrm{BV}}
\newcommand{\h}{\mathrm{H}}
\newcommand{\w}{\mathrm{W}}
\newcommand{\lp}{\mathrm{L}}
\newcommand{\dw}{\mathrm{d}_{W}}

\DeclareMathOperator*{\argmin}{arg\,min}

\usepackage{tikz}
\usepackage{mathdots}
\usepackage{yhmath}
\usepackage{cancel}
\usepackage{color}
\usepackage{siunitx}
\usepackage{array}
\usepackage{multirow}
\usepackage{amssymb}
\usepackage{textcomp, gensymb}          
\usepackage{tabularx}
\usepackage{extarrows}
\usepackage{booktabs}
\usetikzlibrary{fadings}
\usetikzlibrary{patterns}
\usetikzlibrary{shadows.blur}
\usetikzlibrary{shapes}

\newtheorem{theorem}{Theorem}[section]

\newtheorem{proposition}[theorem]{Proposition}
\newtheorem{lemma}[theorem]{Lemma}
\newtheorem{remark}[theorem]{Remark}

\numberwithin{equation}{section}
\usepackage{fancyhdr}
\AtBeginDocument{\RenewCommandCopy\qty\SI}
\ExplSyntaxOn
\msg_redirect_name:nnn { siunitx } { physics-pkg } { none }
\ExplSyntaxOff

\makeatletter
\newcounter{author}
\renewcommand*\author[1]{%
  \stepcounter{author}%
  \ifnum\c@author=1
    \gdef\@author{#1}%
  \else
    \xdef\@author{\unexpanded\expandafter{\@author\and#1}}%
  \fi
  \csgdef{author@\the\c@author}{#1}}
\newcommand*\email[1]{%
  \csgdef{email@\the\c@author}{#1}}
\newcommand*\address[1]{%
  \csgdef{address@\the\c@author}{#1}}
\AtEndDocument{%
  \xdef\author@count{\the\c@author}%
  \c@author=1
  \print@authors}
\newcommand*\print@authors{%
  \ifnum\c@author>\author@count
  \else
    \print@author{\the\c@author}%
    \advance\c@author by 1
    \expandafter\print@authors
  \fi}
\newcommand*\print@author[1]{%
  \par\medskip
  \begin{tabular}{@{}l@{}}%
    \textsc{\csuse{author@#1}}\\
    \csuse{address@#1}\\
    \textit{E-mail address}:
    \href{mailto:\csuse{email@#1}}{\csuse{email@#1}}
  \end{tabular}}
\makeatother

\author{Francesco Colasanto}
\address{DiMaI, Universit\`a di Firenze\\ V.le G.B. Morgagni 67/A, 50134 Firenze,
Italy}
\email{francesco.colasanto@unifi.it}

\author{Pascal Steinke}
\address{Institute for Applied Mathematics, University of Bonn\\Endenicher Allee 60, 53115 Bonn, Germany}
\email{steinke@iam.uni-bonn.de}

\begin{document}
\begin{center}
  {\Large
{Second-Order $\Gamma$-limit for the Cahn--Hillard Functional with Dirichlet Boundary Conditions}}\\[5mm]
{\today}\\[5mm]
Francesco Colasanto and Pascal Steinke
\\[3mm]
\begin{minipage}[c]{0.8\textwidth}
This paper addresses the asymptotic development of order $2$ by $\Gamma$-convergence of the Cahn-Hillard functional with Dirichlet boundary conditions, where the potential has subquadratic growth near the wells. 
\end{minipage}
\end{center}

\section{Introduction}
In this work we analyze the second-order asymptotic development via $\Gamma$-convergence (for the original definition, see \cite{de_giorgi_tullio_75_gamma_convergence}, and for the higher-order definition, see \Cref{sct:gamma_limit}) of the Cahn--Hillard functional $\mathcal{F}_{\eps}^{(0)}:\lp^1(\Omega)\to [0,\infty]$ with Dirichlet boundary condition, given for $ \eps > 0 $ by
\begin{equation}\label{e:Cahn-Hillard functionals definition}
    \mathcal{F}_{\eps}^{(0)}(u) \coloneqq
    \begin{cases}
        \int_{\Omega}W(u)+\eps^2|\nabla u|^2 \dd{x},  & \textup{if  $u\in \h^1(\Omega)$  and $u=g_{\eps}$ on $\partial \Omega$}, \\
        \infty,  & \textup{otherwise}.
    \end{cases}
\end{equation}
In \eqref{e:Cahn-Hillard functionals definition}, the function $W:\R \to [0,\infty)$ is a continuous double-well potential, which means that $ W( x) = 0 $ holds if and only if $ x \in \{ a, b \} $ with $a<b$ and sub-quadratic growth near the wells. $\Omega$ is a $C^2$ open  and bounded subset of $\R^n$, while $g_{\eps}\in \h^{1/2} ( \partial \Omega )$ for every $\eps\in (0,\infty)$.

The Cahn--Hilliard functional 
\begin{equation}\label{e:Modica-Mortola}
    E^{(1)}_{\eps}(u)=\begin{cases}
        \int_{\Omega}\frac{1}{\eps}W(u)+\eps|\nabla u|^2 \dd{x},  & \textup{if $u\in \h^1(\Omega)$}, \\
        \infty , & \textup{otherwise},
    \end{cases}
\end{equation}
occurs in the modelling of many physical phenomena, from phase transition and separation (cf. \cite{AllenCahn79,CahnHillard58,CahnElliotNovick,BloweyElliott}) to processes of nucleation and coarsening (cf. \cite{CahnHillard59,BatesFife}), where $a$ and $b$ represent the \textit{phases} of the mechanical system.
It was first studied in \cite{ModicaMortola77} by Modica and Mortola, were they proved $ \Gamma $-convergence to the perimeter functional
\begin{equation}\label{e:Gammaconv modica mortola without constraint}
    E^{(1)} (u)\coloneqq \begin{cases}
     \frac{C_W}{b-a}|\diff u|(\Omega), & \textup{if  $u\in \bv(\Omega;\{a,b\})$}, \\
        \infty, & \textup{otherwise},
    \end{cases}
\end{equation}
for a constant $ C_W$ depending on the potential $ W $. 
In the subsequent work \cite{Modica87}, analogous results were obtained in the \textit{mass constrained} case, see also \cite{Sternberg88,Owen88,FonsecaTartar}, and \cite{baldo_90_multiphase} for the multiphase case.

The \textit{Dirichlet} case is more involved and was first studied in \cite{RubinsteinSternberg90} under the assumption that $W\in C^2(\R)$ with $W''(a),W''(b)>0$ (which equates to quadratic growth at the wells) and suitable boundary data $g_{\eps}$ (see \cite[Sct.~2]{RubinsteinSternberg90}). The authors showed that the first-order $ \Gamma $-limit is given by
\begin{equation}\label{e:first order gamma limit dirchlet case}
    \Gamma\textup{-}\lim_{\eps\to 0} \frac{\mathcal{F}_\eps^{(0)}(u)}{\eps} = \begin{cases}
     \frac{C_W}{b-a}|\diff u|(\Omega) +\int_{\partial \Omega} \textup{d}_W(u,g)\dd{\hm^{n-1}},
     & \textup{if $u\in \bv(\Omega;\{a,b\})$}, \\
        \infty , & \textup{otherwise},
    \end{cases}  
\end{equation}
for $\mathcal{F}_{\eps}^{(0)}$ defined in \eqref{e:Cahn-Hillard functionals definition}. Here $g$ is the limit of $g_{\eps}$ in $\lp^1(\partial \Omega)$ as $\eps\to 0$ and $\dw $ is the geodesic distance defined by 
\begin{equation}\label{e:geodesic distance}
   \dw(r,s)\coloneqq
       \abs{\int_r^s 2\sqrt{W(\rho)}\dd{\rho} }.
\end{equation}
The inhomogeneous has been covered in \cite{ansini_braides_chiad_03_gradient_theory}, who also proved a \emph{fundamental estimate} for the first-order functional, which is currently not available for the second-order functional.
In \cite{gazoulis_24_gamma_convergence} the author used fundamental estimate for an explicit representation of the limiting functional in the multiphase case. 

The first work covering the second-order asymptotic development for Cahn--Hillard like functionals is \cite{DalMasoFonsecaLeoni15}, where the authors proved the triviality of the second order $\Gamma$-limit in the mass-constrained  subquadratic case with the addition of constant Dirichlet data $ 1$, crucial for the use of classical symmetrization techniques in their proof. In \cite[Thm.~1.1, Thm.~1.2]{Leoni-Murray15} and \cite[Thm.~1.3]{leoni_murray_19_second_order_2}, the authors removed the Dirichlet boundary condition and proved the non-triviality of the second order $\Gamma$-limit, both for quadratic and sub-quadratic wells.  

In \cite{fonseca2025secondordergammalimitcahnhilliardfunctional1} and \cite{fonseca2025secondordergammalimitcahnhilliardfunctional} the authors study the second order asymptotic development of the functionals $\mathcal{F}_{\eps}^{(0)}$, for potentials $W$ with quadratic growth near the wells, assuming that the unique minimizer of the first order functional is the constant function $ b $. 
In the first work \cite{fonseca2025secondordergammalimitcahnhilliardfunctional1} the authors crucially assume that the boundary data is bounded away from one of the wells, which means that there exists some $ \alpha_-\in (a,b] $ such that
\begin{equation}
\label{eq:ass_paper_1}
    a<\alpha_- \leq g_\eps(x)\leq b.
\end{equation}
Note that this assumption does not automatically imply that the constant function $ b $ is the unique minimizer of the first order functional (\ref{e:first order gamma limit dirchlet case}), see \Cref{fig:counterexample}.

\begin{figure}
\label{fig:counterexample}

\tikzset{every picture/.style={line width=0.75pt}} 

\begin{tikzpicture}[x=0.75pt,y=0.75pt,yscale=-1,xscale=1]
uncomment if require: \path (50,300); 

\draw  [fill={rgb, 255:red, 80; green, 227; blue, 194 }  ,fill opacity=1 ] (317,135.1) .. controls (340.2,133.9) and (377.8,127.9) .. (376.2,144.7) .. controls (374.6,161.5) and (359,160.85) .. (298,161.25) .. controls (237,161.65) and (192.5,170.85) .. (192.5,145.25) .. controls (192.5,119.65) and (293.8,136.3) .. (317,135.1) -- cycle ;
\draw  [fill={rgb, 255:red, 126; green, 211; blue, 33 }  ,fill opacity=1 ] (319.8,60.7) .. controls (339.8,50.7) and (447.2,41.3) .. (455,54.3) .. controls (462.8,67.3) and (434.6,225.1) .. (416.6,245.5) .. controls (398.6,265.9) and (320,182.9) .. (322.6,153.1) .. controls (325.2,123.3) and (299.8,70.7) .. (319.8,60.7) -- cycle ;
\draw [color={rgb, 255:red, 208; green, 2; blue, 27 }  ,draw opacity=1 ][line width=2.25]    (321.94,135.45) .. controls (323.91,150.41) and (321,143.86) .. (323.08,160.02) ;

\draw (250,137.9) node [anchor=north west][inner sep=0.75pt]    {$u=a$};
\draw (364.17,125.73) node [anchor=north west][inner sep=0.75pt]    {$u=b$};
\draw (185,166.4) node [anchor=north west][inner sep=0.75pt]    {$|g -a|\ll 1\ $};
\draw (461.67,129.73) node [anchor=north west][inner sep=0.75pt]    {$|g-b|\ll 1$};

\end{tikzpicture}

\caption{If the boundary data $ g $ is close to $ a $ on a piece of large surface, but small volume, then a transition at the red line is energetically more favourable than no transition.}

\end{figure}
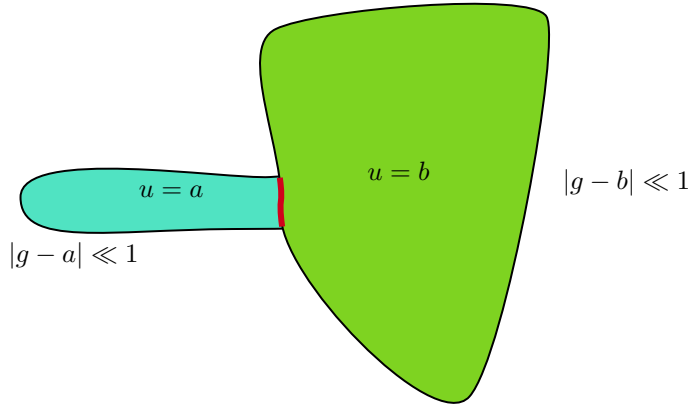

The first non-trivial rescaling is then of order $ \eps $ and the limiting functional is given by
\begin{equation*}
    \mathcal{F}^{(2)} ( b ) 
    \coloneqq
    \int_{ \partial \Omega }
    -\kappa ( y ) 
    \int_0^\infty 
    2 \sqrt{W ( z_{g(y)} ( s ) ) } z_{g(y)}' ( s ) s 
    \dd{ s }
    \dd{ \hm^{n-1} ( y ) },
\end{equation*}
where $ \kappa $ is the mean curvature of $ \partial \Omega $. Here, for a given $ \alpha \in \R $, the function $ z_\alpha $ is defined as the unique solution to the Cauchy problem
\begin{equation}
    \label{eq:cauchy_problem}
    \begin{cases}
        z_\alpha ' ( s )
        = 
        \sqrt{ W ( z_\alpha ( s ) )},
        \\
        z_\alpha ( 0 ) = \alpha,
    \end{cases}
\end{equation}
see also \Cref{prop:cauchy_problem}.
Note that the function $ z_\alpha ( \cdot / \eps ) $ is the optimal transition profile on the scale $ \eps $ between $ \alpha $ and $ b $ and is used in the construction for the recovery sequence in previous works for the first order $ \Gamma $-limit.
In the subsequent work \cite{fonseca2025secondordergammalimitcahnhilliardfunctional}, the assumption (\ref{eq:ass_paper_1}) is removed, which leads to a rescaling of order $ \log(1/\eps) \eps $ and a different limiting energy, which only sees parts of the boundary with negative mean curvature.

In this paper we want to focus on the case of a potential $ W $ with subquadratic growth at the wells, by which we mean that there exists some $ q \in (0,1) $ such that $ W ( x ) \sim \abs{x-a}^{1+q}$ for $ \abs{x-a} \ll 1$ and $ W ( x ) \sim \abs{x-b}^{1+q} $ for $ \abs{ x - b } \ll 1 $. 
The key difference to the case of quadratic growth is that solutions to the Cauchy problem (\ref{eq:cauchy_problem}) reach the well $ b $ in finite time, which can be explicitly written as
\begin{equation*}
    T^{(\alpha)} = \int_\alpha^b \frac{1}{\sqrt{W(s)}} \dd{ s }.
\end{equation*}
Since the solution in the quadratic case only approaches $ b $ with exponential speed, but never reaches the well, the construction of the recovery sequence in the subquadratic case simplifies. Moreover we were able to explicitly analyze the behaviour of minimizers of the weighted one-dimensional Cahn-Hilliard energy by carefully studying the corresponding Euler--Lagrange equation, see \Cref{lem:properties_of_1d_minimizers}, which is a major simplification compared to \cite{fonseca2025secondordergammalimitcahnhilliardfunctional1}.
In the end, we are able to show that on the scale $ \varepsilon $, we get a non-trivial second-order $ \Gamma $-limit. This is a surprising difference to the quadratic case, where regions of the boundary with boundary data $ g = a $ exhibit an $ \eps \log ( 1/\eps ) $ scaling, and regions with $ g > a + \delta $ an $ \eps $-scaling. The limiting functional however stays the same, see \Cref{prop:recovery_seq_second_order} and \Cref{thm:liminf_second_order}, with less strict assumptions than in the quadratic case. 
The reason for the different scaling is that in the quadratic case, additional energy is necessary to move the optimal transition profile out of the well since the unique solution $ z_a $ of the Cauchy problem (\ref{eq:cauchy_problem}) is constant. In the subquadratic case however, where the solutions are not unique, there exists a solution which immediately moves out of the well.

We also want to comment on the assumptions we make. Currently we have certain smoothness and regularity assumptions on $ g $ and $ g_\eps $ (see \Cref{sct:assumptions}) which prohibit boundary data of the form $ g = a \chi_A + b \chi_{ \partial \Omega \setminus A } $. Even though we do not rely on the regularity of $ g $ and could in principle only assume $ g \in \lp^1 ( \partial \Omega ) $, the assumptions
\begin{equation}
\label{eq:necessary_assumptions}
    \int_{ \partial \Omega}
    \frac{ \dw ( g_\eps , g ) }{\eps }
    \dd{ \hm^{n-1} }
    \to 0 
    \quad \text{and} \quad
    \eps \int_{ \partial \Omega } \abs{ \nabla_\tau g_\eps }^2
    \dd{ \hm^{n-1} }
    \to 
    0
\end{equation}
are necessary in order to apply our slicing techniques. 
Since $ g_\eps $ is the boundary data of an $ \h^1 ( \Omega ) $ function, it has to have at least $ \h^{1/2}( \partial \Omega ) $-regularity, thus it is necessary that $ g_\eps $ is a regularization of $ g $.
By flattening and applying a standard mollifier $ g_\eps = \eta_{\rho_\eps} \ast g $, we note that $ \rho_\eps $ would have to satisfy $ \rho_\eps / \eps \to 0 $ to satisfy the first assumption in (\ref{eq:necessary_assumptions}). For the second assumption in (\ref{eq:necessary_assumptions}) however, we need the reverse convergence $ \eps / \rho_\eps \to 0 $, which is not feasible.

The strategy of the proof is similar as in \cite{fonseca2025secondordergammalimitcahnhilliardfunctional1}. By using the Area Formula and Fubini, we reduce the $ n$-dimensional functional to a weighted one-dimensional functional by writing
\begin{equation*}
    \int_{ \{ x \in \Omega \colon \mathrm{dist} ( x, \partial \Omega ) \leq \delta \} }
    f ( x ) 
    \dd{ x }
    =
    \int_{ \partial \Omega }
    \int_0^\delta 
    f(  \Phi (y, t ) )
    \det J_\Phi ( y, t )
    \dd{ t }
    \dd{ \hm^{n-1 } ( y ) },
\end{equation*}
where $ \Phi (y, t ) \coloneqq y + t \nu ( y ) $ for the inner unit normal $ \nu $. We construct the recovery sequence by choosing optimal transition profiles in direction $ \nu $. For the $\liminf$-inequality, we have to carefully study minimizers of the weighted one-dimensional functional 
\begin{equation*}
    \int_0^{\delta } (W ( v ) + \eps^2 ( v')^2 ) \omega \dd{ t }.
\end{equation*}
Since the second-order functional will not be non-negative, we will have to quantify the lower bounds for the one-dimensional functional in order to later justify pulling the limit inferior into the integral over the boundary of $ \Omega $.

A future question of interest is what the second-order $ \Gamma $-limit looks like for less regular $ g $. Our current conjecture is that if $ g $ is of the form $ a \chi_A + b \chi_{\partial \Omega \setminus A } $, the limiting functional, with the same scaling $ \eps $, is proportional to the size of the jump set $ J_g$. 
However a different construction for the recovery-sequence will be necessary. The most interesting question remains what the second-order functional looks like when not assuming that $ b $ is the unique minimizer, which has for example been discussed in \cite[Rmk.~1.3]{fonseca2025secondordergammalimitcahnhilliardfunctional1}.

The structure of the paper is as follows. In \Cref{sct:prelim}, we go over the definition of higher-order $ \Gamma $-limits and the assumptions on the potential, the domain and the boundary data. We continue in \Cref{sct:1_dim} with the $ \Gamma $-limit in the one-dimensional weighted case, which is the core of this work. Then we show in \Cref{sct:multidim_case} how we can use slicing to transfer these results to the multidimensional case. In \Cref{sct:appendix}, we include those proofs which are very similar to already existing literature and thus of lesser interest to the reader.

\section{Preliminaries} 
\label{sct:prelim}

\subsection{Higher-order Gamma-limits}
\label{sct:gamma_limit}

We want to start by discussing the asymptotic development of functionals via $\Gamma$-convergence introduced in \cite{anzellotti_baldo_93_asymptotic_gamma_convergence}. Given a metric space $X$, a family of functionals ${F}_{\eps}:X\to [-\infty,\infty]$ and of nonnegative real numbers $\delta_{\eps}$  for $\eps>0$, the \textit{first order asymptotic development} of ${F}_{\eps}$ \textit{with respect to} $\delta_{\eps}$ \textit{on} $X$ is
\begin{equation*}
{F}_{\eps}={F}^{(0)}+\delta_\eps {F}^{(1)} +o(\delta_\eps) \quad  \textup{on $X$}
\end{equation*}
if and only if $\delta_{\eps}\to 0$ as $\eps\to 0$ and  ${F}^0,{F}^1:X\to [-\infty,\infty]$ are such that 
\begin{equation*}
      \textup{$\Gamma$-$\lim_{\eps\to 0} {F}_{\eps}={F}^{(0)}$ on $X$, \; \;   $m_0:=\inf_X {F}^{(0)}\in \R$ \; \; and \; \;  $\Gamma$-$\lim_{\eps\to 0} {F}^{(1)}_{\eps}={F}^{(1)}$ on $\mathcal{U}^0$}
\end{equation*}
where 
\begin{equation}\label{e:first order apporox functionals}
   {F}_{\eps}^{(1)}:=\frac{{F}_{\eps}-m_0}{\eps} 
\end{equation}
and $\mathcal{U}^0:=\{x\in X \; : \; {F}^{(0)}(x)=m_0\}$. More generally, if $j\in \N\setminus\{1\}$ and $\{\delta^{(1)}_\eps,\dots,\delta^{(j)}_\eps \}$ are families of nonnegative real numbers for $\eps>0$, we say that the \textit{$j$th-order asymptotic development of ${F}_{\eps}$} \textit{with respect to} $\{\delta^{(1)}_\eps,\dots,\delta^{(j)}_\eps \}$ \textit{on} $X$ is
\begin{equation}\label{e:first asymptotic exapnsion gamma conv.}
{F}_{\eps}={F}^{(0)}+\delta_{\eps}^{(1)}{F}^{(1)}+\cdots+ \delta^{(j)}_\eps {F}^{(j)}+ o(\delta^{(j)}_{\eps})  \quad \textup{on $X$}
\end{equation}
if and only if  $\sigma^{(k)}:=\delta_{\eps}^{k+1}/\delta_{\eps}^{k}\to 0 $ as $\eps \to 0$, for every $k=1,\dots,j-1$ and ${F}^{(0)},{F}^{(1)},\dots,{F}^{(j)}:X\to [-\infty,\infty]$ are such that 
\begin{equation*}
{F}^{(k)}_{\eps}={F}^{(k)}+\delta_{\eps}^{(k+1)}{F}^{(k+1)}+o(\delta^{(k+1)}_{\eps})  \quad \textup{on $\mathcal{U}^{k-1}$}
\end{equation*}
for every $k=0,\dots,j-1$, where $\mathcal{U}^{-1}:=X$, ${F}^{(0)}_{\eps}:={F}_{\eps}$, 
\begin{equation}\label{e:def set of k-minimizers}
m_k:=\inf_{\mathcal{U}^{k-1}}{F}^{(k)}, \quad  \quad 
        \mathcal{U}^k:=\{x\in \mathcal{U}^{k-1} \; : \; {F}^{(k)}(x)=m_k\}
\end{equation}
and 
\begin{equation*}
   {F}^{(k+1)}_{\eps}:=\frac{{F}^{(k)}_{\eps}-m_{k}}{\sigma^{(k)}_{\eps}} \quad \quad \textup{for every $k=0,\dots,j-1$.}
\end{equation*}
This notion of asymptotic expansion not only provides a criterion for selecting minimizers of ${F}^{(0)}$, because every limit point of minimizers of the family ${F}_{\eps}$ has to be a minimizer of ${F}^{k}$ on $\mathcal{U}^{k-1}$, recalling that $\mathcal{U}^{j}\subseteq \mathcal{U}^{j-1} \subseteq  \cdots \subseteq \mathcal{U}^0 \subseteq \mathcal{U}^{-1}=X$, but also gives the following asymptotic formula for the minima
\begin{equation}\label{e:asymptotic expansion of minima}
    m_{\eps}=m_0+\delta^{(1)}_{\eps}m_1+\cdots +\delta^{(j)}_{\eps}m_{j}+o(\delta^{(j)}_\eps) 
\end{equation}
where $m_{\eps}:=\inf_X {F}_{\eps}$.

\subsection{Assumptions}
\label{sct:assumptions}
As mentioned in the introduction, we are concerned with functional (\ref{e:Cahn-Hillard functionals definition}).
Throughout the paper, we will assume that $ \Omega \subseteq \R^n $ is open and bounded with $ C^2$-boundary. Moreover we want to ignore the issues of possible escape to infinity and assume throughout the paper that the boundary data $ g_\eps $ maps to the interval $ [a,b] $, so that we can by a truncation argument assume that $ u_\eps $ maps to $ [a,b] $ as well. For the potential $ W $, we make the following two assumptions.
\begin{enumerate}[label=(\roman*)]
\item $W^{-1}({0})=\{a,b\}$ with $a<b$ and $W\in C^2(\R\setminus \{a,b\})$.
\item There exists $q\in (0,1)$  such that 
\begin{equation}\label{e:limite subquadratico}
    \lim_{s\to a}W''(s)|s-a|^{1-q} \in (0, \infty ) \quad\text{ and }\quad \lim_{s\to b}W''(s)|s-b|^{1-q}\in (0,\infty ). 
\end{equation}
\end{enumerate}
For the boundary data, we assume that $ g \in C^2 ( \partial \Omega ) $ such that $ b $ is the unique minimizer of the perimeter functional (\ref{e:first order gamma limit dirchlet case}). As a consequence (see \Cref{prop:curvature_b-unqie_min}), we have that $ \{ g = a \} \subseteq \{ \kappa \leq 0 \} $. Due to limitations of the techniques we apply in this paper, we will have to make the stronger assumption that $ \{ g=a \} \subseteq \{ \kappa < 0 \} $, and due to the continuity of $ g $, we consequently already have that the inclusion holds uniform, which means that there exists $ \alpha_- > a $ and $ \kappa_0 < 0 $ such that 
\begin{equation} 
\label{eq:uniform_inclusion}
\{ g \leq \alpha_-\} \subseteq \{ \kappa < \kappa_0 \}.
\end{equation}
The problem when trying to remove this assumption is the following. For the $ \Gamma $-lower bound, we reduce the domain via slicing to a $ 1 $-dimensional interval $ [0, \delta ]$. To argue that an immediate transition from $ g_\eps $ to $ b $ at $ 0 $ here is already optimal, we have to control the boundary value at $ \delta $ and argue that it is already close to $ b $, which has been proven in \cite[Thm.~4.9]{fonseca2025secondordergammalimitcahnhilliardfunctional1}. The proof relies crucially on a result by Sternberg and Zumbrum \cite{sternberg_zumbrun_98_connectivity_of_phase_boundaries} and on a result by Caffarelli and Córdoba \cite{caffarelli_cordoba_95_unfirom_convergence_of_a_singular_perturbation_problem} who show uniform convergence of phase boundaries in Cahn--Hilliard like perturbation problems. The latter result is however not quantitative since its proof relies on a contradiction argument, meaning that given $ \delta > 0 $, we only know that $ \abs{u_\eps ( \delta )-b} \ll 1 $ for $ \eps > 0 $ sufficiently small in an unquantified dependence on $ \delta $. The inclusion (\ref{eq:uniform_inclusion}) ensures that for $ \delta > 0 $ \emph{fixed} and sufficiently small, the slices will still have negative curvature, which will be crucial to our arguments.

We assume that $ g_\eps \in \h^{1}( \partial \Omega ) $ is a suitable approximation of $ g $ in the sense that there exist $ a < \alpha_-  < b $ and $ \kappa_0 < 0 $ such that for all $ \eps > 0 $ sufficiently small, we have for $ \hm^{n-1}$-almost every $ y \in \partial \Omega $ that
\begin{equation}
\label{eq:assumption_bdry_data_distinction}
    g_\eps ( y ) \in [\alpha_-, b]
    \quad \text{or} \quad
    \kappa(y) \leq \kappa_0.
\end{equation}
If $ g_\eps $ approximates $ g $ suitably, this is a consequence of inclusion (\ref{eq:uniform_inclusion}).

Moreover, we assume that $ g_\eps $ converges to $ g $ fast enough, but its derivative does not explode too fast such that $ g_\eps \to g $ $ \hm^{n-1}$-pointwise almost everywhere,
\begin{align}
    \label{eq:geps_to_g}
    \int_{ \partial \Omega }
    \frac{
    \dw ( g_\eps ( y ) , g ( y ) ) }{\eps}
    \dd{ \hm^{n-1 } ( y ) }
    & \to 0 
    \shortintertext{and }
    \label{eq:ass_tang_deriv_squared}
    \eps \int_{ \partial \Omega}
    \frac{\abs{ \nabla_\tau g_\eps ( y ) }^2 }{W(g_\eps ( y ) ) } 
    \dd{ \mathcal{H}^{n-1} ( y ) }
    & \to 0,
\end{align}
where $ \nabla_\tau $ is the tangential gradient on $ \partial \Omega $. For the last assumption, we want to comment that it is in particular satisfied if $ g_\eps \in \h^2 ( \partial \Omega ) $ such that $ \norm{ \nabla^2 g_\eps }_{ \lp^\infty ( \partial \Omega ) } $ stays bounded and 
\begin{equation}
\label{eq:ass_tang_der_weakened}
    \eps \int_{ \partial \Omega } \abs{ \nabla_\tau g_\eps ( y ) }^2 \dd{ \hm^{n-1 } ( y ) } \to 0.
\end{equation}
To see this we distinguish the area where $ g_\eps $ is close to $ \{a,b\} $ and the area where $ g_\eps $ is away from $ \{a,b\} $. If $ g_\eps $ is away from $ \{a,b\} $, then $ W ( g_\eps ( y ) ) \gtrsim 1 $ and the convergence (\ref{eq:ass_tang_der_weakened}) already implies (\ref{eq:ass_tang_deriv_squared}). On the other hand if $ g_\eps (y ) $ is close to $ \{a,b\} $, then by flattening the boundary, applying the fundamental theorem of calculus and using $ g_\eps \in [a,b]$, we obtain that the integrand  in (\ref{eq:ass_tang_deriv_squared}) stays bounded.

Recall the definition (\ref{e:geodesic distance}) of the geodesic distance.
As a consequence of the growth at the wells, we make the following observations for the asymptotic behaviour of $ W $, $ W^{-1} $ and $ \dw $ close to the wells.
\begin{lemma}\label{p:prop W}
There exists $\sigma\in (0,1)$ and $ \delta > 0 $ such that 
\begin{alignat}{3}
\label{e:subquadratic in a}
    \sigma|s-a|^{1+q} &\leq W(s)\leq \sigma^{-1}|s-a|^{1+q} \quad &&\text{for every $s\in [a,a+\delta]$},
    \\
\label{e:subquadratic in b}
    \sigma|s-b|^{1+q}&\leq W(s)\leq \sigma^{-1}|s-b|^{1+q} \quad &&\text{for every $s\in [b-\delta,b]$},
    \\
    \notag
    \sigma \abs{ s - a }^{(3+q)/2}
    & \leq
    \dw ( a , s ) 
    \leq
    \sigma^{-1}
    \abs{s - a }^{(3+q)/2} \quad &&\text{for every } s \in [a, a + \delta ]
    \text{ and}
    \\
    \notag
    \sigma \abs{ s - b }^{(3+q)/2}
    & \leq
    \dw ( s , b ) 
    \leq
    \sigma^{-1}
    \abs{s - b }^{(3+q)/2} \quad &&\text{for every } s \in [b- \delta, b] 
    .
\end{alignat}
Furthermore there exists $ \rho > 0 $ such that $ W^{-1} \colon [0, \rho ] \to [ b- \delta, b ]$ is well-defined, decreasing and satisfies
\begin{equation*}
    \sigma t^{1/(1+q)}
    \leq
    b - W^{-1} ( t ) 
    \leq
    \sigma^{-1} t^{1/(1+q)}
    \quad \text{for all } t \in  [0, \rho ].
\end{equation*}
Finally $1/\sqrt{W} \in \lp^1_{\textup{loc}}(\R)$ and $W'\in C^{0,q}_{\textup{loc}}(\R)$.
\end{lemma}

\begin{proof}
From L'Hôpital's rule, the continuity of $W$, $W^{-1}(\{0\})=\{a,b\}$ and (\ref{e:limite subquadratico}) we can deduce \eqref{e:subquadratic in a} and \eqref{e:subquadratic in b}. In particular $1/\sqrt{W}\in \lp^1_{\textup{loc}}(\R)$. Due to \eqref{e:limite subquadratico}, there exists $\mu\in (0,\infty)$ and $\eps\in (0,\infty)$ such that 
\begin{equation*}
    0\leq W''(b-t)\leq \frac{\mu}{|t|^{1-q}}
\end{equation*}
for every $ t\in (-\eps,\eps)\setminus \{0\}$. Therefore, for every $s_1,s_2\in (b-\eps,b)$ with $s_1<s_2$, thanks to the subadditivity of $x\mapsto x^q$, we have that 
\begin{equation*}
    |W'(s_2)-W'(s_1)|\leq \int_{s_1}^{s_2}W''(s)\dd s\leq \mu\int_{t_2}^{t_1}\frac{1}{t^{1-q}} \dt \leq \mu(t_1^{q}-t^q_2)\leq \mu(t_1-t_2)^q=\mu(s_2-s_1)^q
\end{equation*}
where $t_1=b-s_1$ and $t_2=b-s_2$. In particular, arguing in the same way in the right and left neighbourhood of $a$ and $b$, by $W'\in C^1(\R\setminus \{a,b\})$ we can conclude that $W'\in C^{0,q}_{\textup{loc}}(\R)$. 
The other claims are direct computations.
\end{proof}
Lastly we recall from the theory of ordinary differential equations the following Proposition regarding existence of solutions to the Cauchy problem.
\begin{proposition}
\label{prop:cauchy_problem}
For every $\alpha,x \in [a,b]$ we define 
\begin{equation}
    \label{eq:def_Psialpha}
    \Psi_\alpha ( x )
    \coloneqq
    \int_\alpha^x
    \frac{1}{\sqrt{W(s)}}
    \dd{ s }.
\end{equation}
Let $T^{(\alpha)} \coloneqq \Psi_\alpha ( b ) $ and $ z_\alpha \colon [0,\infty) \to [a,b] $ be defined by
\begin{equation*}
    z_\alpha ( t )
    \coloneqq
    \begin{cases}
        \Psi_\alpha^{-1} ( t ), & \text{if } t \in [0, T^{(\alpha ) }],
        \\
        b, &\text{if } t \in (T^{(\alpha)}, \infty ).
    \end{cases}
\end{equation*}
Then $z_{\alpha} \in C^1 ( [0, \infty ) )$ solves
\begin{equation}
\label{eq:cauchy_problem_full}
\begin{cases}
      z'_{\alpha}(t)=\sqrt{W(z_{\alpha}(t))}, & \textup{for every }t\in \R, \\
      z_{\alpha}(0)=\alpha,
      \\
      z_{\alpha }(t) \in [a, b ], & \textup{for every } t \in \R .
\end{cases}
\end{equation}
\end{proposition}
Note that the solution to equation (\ref{eq:cauchy_problem_full}) is unique if and only if $ \alpha > a $ due to the last condition. For $ \alpha = a $, we choose the unique solution which immediately leaves the well $ a $.

\section{1-dimensional weighted case}
\label{sct:1_dim}
As we will later see, the functional applied to functions on a $d$-dimensional domain will be reduced to a functional acting on functions with one-dimensional domains, at the cost of introducing a weight $ \omega $, which comes from the Jacobian $ \omega \coloneqq \det J_\Phi $ of the transformation map $ \Phi $, see \Cref{p:The Phi}.
Thus for an interval $ I = (0,T)$, we are interested in the behaviour of the functional
\begin{align*}
    G_\eps^{(0)} \colon \lp^1(I) &\to [0, \infty ]
    \\
    v & \mapsto
    \begin{cases}
        \int_0^T (W ( v (t ) ) + \eps^2 (v'(t))^2 ) \omega ( t )
        \dd{ t },
        &\text{if } v \in \h^1 ( I ) \text{ with } v(0)= \alpha_\eps \text{ and } v(T)= \beta_\eps,
        \\
        \infty,
        & \text{otherwise}.
    \end{cases}
\end{align*}
Here $ \omega \in C^1 ( [0,T]) $.
In the application later, $ \alpha_\eps =g_\eps ( y ) $ for a fixed point on the boundary $ y \in \partial \Omega $, and $ \beta_\eps $ will be exponentially close to $ b $. Thus it follows from assumption (\ref{eq:assumption_bdry_data_distinction}) and $ w'(0) = - \kappa ( y ) $ (see \Cref{p:The Phi}) that we always assume that there exists $ a < \alpha_- < b $ and $ \kappa_0 < 0 $ such that $ \beta_\eps \in [ \alpha_-, b ] $ and 
\begin{equation}  
\label{eq:bdry_assumptions_1d}
\alpha_\eps\in [\alpha_-,b ] 
    \quad \text{or}\quad \omega' ( 0 ) \geq -\kappa_0.
\end{equation}
Moreover for some $ \alpha \in [a,b] $ we have convergence of the boundary values
\begin{equation*}
    \alpha_\eps \to \alpha  \quad \text{and} \quad
    \beta_\eps \to b \quad \text{as } \eps \to 0.
\end{equation*}
Since $ [0,T] $ is a slice close to the boundary, the interval length $ T $ will be very short and $ \omega $ thus almost constant. We thus may assume $ T \leq 1 $ and that in the case $ \omega'(0) \geq - \kappa_0 $, the interval length $ T $ is so small that $ \omega' ( t ) \geq - \kappa_0/2 $ for all $ t \in [0,T] $ so that by the fundamental theorem of calculus, we have
\begin{equation}
    \label{eq:weight_lower_bound}
    \omega ( t ) \geq \omega ( 0 ) - \frac{\kappa_0}{2} t 
    \quad
    \text{for all } t \in [0,T] \text{ if } \omega'(0) > - \kappa_0.
\end{equation}
Furthermore defining $ \omega_0 \coloneqq \min_{[0,T]} \omega $ and $ \omega_1 \coloneqq \max_{[0,T] }\omega $, we note that $ \omega_1 - \omega_0 $ will be small with $ \omega (0)=1$. 
All constants $ C $ (which may change from line to line) in this section only depend on $ W $, $ \alpha_{-} $, $ \norm{ \omega }_{ C^1 [0, T] } $ and $ \kappa_0 $. If we say that a statement holds for all $ \eps > 0 $ sufficiently small, we mean that there exists $ \eps_0 > 0 $ possibly depending on $ W $, $ \alpha_- $, $ \norm{ \omega }_{ C^1 [0,T]} $, $ \kappa_0 $ and the convergence speed of $ \beta \to b $ such that the statement holds for all $ 0 < \eps < \eps_0$.
\begin{proposition}
Let $G^{(0)}:\lp^1(I)\to [0,\infty]$ be the functional given by
\begin{equation*}
    G^{(0)}(v)=\int_0^T W(v(t))\omega(t) \dd{t}.
\end{equation*}
Then we have that 
\begin{equation*}
    \Gamma \textup{-}\lim_{\eps \to 0} G_{\eps}^{(0)}=G^{(0)}.
\end{equation*}
\end{proposition}
\begin{proof}
    The proof is identical to the unweighted case, see \Cref{prop:zero_order_gamma_limit}.
\end{proof}

\begin{lemma}
\label{prop:minimizer_properties}
  For all $ \eps > 0 $, the functional $G_\eps^{(0)} $ assumes its infimum.
  Moreover for every minimizer $ v $, we have $v\in C^2([0,T])$ with $v(0)=\alpha_{\eps}$, $v(T)=\beta_{\eps}$, $a\leq v \leq b$ and 
  \begin{equation}\label{e:E-L eq}
      v''(t)=\frac{W'(v(t))}{2\eps^2}-\frac{\omega'(t)}{\omega(t)}v'(t)
  \end{equation}
  for every $t\in [0,T]$. Lastly we have the energy bound
  \begin{equation}
  \label{eq:energy_bound}
      \min G_\eps^{(0)}
      \leq
      \eps \omega_1 \left( \dw ( \alpha_\eps , b ) + \dw ( \beta_\eps , b ) \right)
  \end{equation}
\end{lemma}

\begin{proof}
The existence of minimizers can be shown with the direct method. Indeed if $(v_k)$ is a minimizing sequence of $G_{\eps}^{(0)}$, we can assume that for every $k\in \N$ 
\begin{equation*}
    \int_0^T W(v_k(t))+\eps^2(v_k'(t))^2\dt \leq M
\end{equation*}
for some constant $M>0$, thus the sequence $(v_k)$ is bounded in $\h^1(I)$. By truncation we may assume that $ v_k $ maps to $ [a,b]$. 
Therefore there exists $v\in \h^1(I)$ with $v(0)=\alpha_{\eps}$ and $v(T)=\beta_{\eps}$ such that $v_k\to v$ uniformly on $[0,T]$ and $v'_k\weakto v'$ in $\lp^2(I)$, up to a subsequence. Thus, by the superadditivy of the limit inferior and the convexity of the function $x\mapsto x^2$, we conclude that $v\in \argmin G_{\eps}^{(0)}$.

Fix now $v\in \argmin G_{\eps}^{(0)}$. Thanks to a truncation argument we have that $a\leq v\leq b$. Through outer variations, we obtain that
\begin{equation}\label{e:E-L primo ordine}
    \int_0^T (W'(v)\phi+ 2\eps^2v'\phi')\omega \dd{t}=0,
\end{equation}
for every $\phi\in C^{\infty}_c(I)$, therefore the distributional derivative of $2\eps^2v'\omega$ is $W'(v)\omega$. Thus $v'\omega\in \h^1(I)$ and using the fundamental theorem, we deduce $v'\omega\in C^1([0,T])$. To finish we observe that since $ 1/\omega \in C^1([0,T]) $, we obtain $v\in C^2([0,T])$ and that equation \eqref{e:E-L eq} holds.

We now prove the energy bound (\ref{eq:energy_bound}). 
Fix $ \eps > 0 $ and recalling \Cref{prop:cauchy_problem}, define $ T_{\eps } \coloneqq \eps \Psi_{\alpha_\eps } ( b ) $ and $ S_\eps \coloneqq \eps \Psi_{\beta_\eps } ( b  ) $.We note by the local integrability of $ W^{-1/2}$ (see \Cref{p:prop W}) that $S_\eps, T_\eps \leq C \eps $
for some $C>0$.
In particular there exists $\eps_1\in (0,\infty)$ such that $T_{\eps}<T-S_{\eps}$ for every $\eps\in (0,\eps_1)$. 
Fix $\eps\in (0,\eps_1)$ and define the function $v_{\eps}\colon[0,T]\to [a,b]$ as
\begin{equation}
\label{eq:def_recov_seq}
v_{\eps}(t)\coloneqq
\begin{cases}
z_{\alpha_\eps}(t/\eps), & \textup{if $t\in [0,T_{\eps}]$}, \\
b, & \textup{if $t\in (T_{\eps},T-S_{\eps})$}, \\
z_{\beta_\eps}((T-t)/\eps) , & \textup{if $t\in [T-S_{\eps},T]$}, 
\end{cases}
\end{equation}
see also \Cref{fig:plot_recovery}.
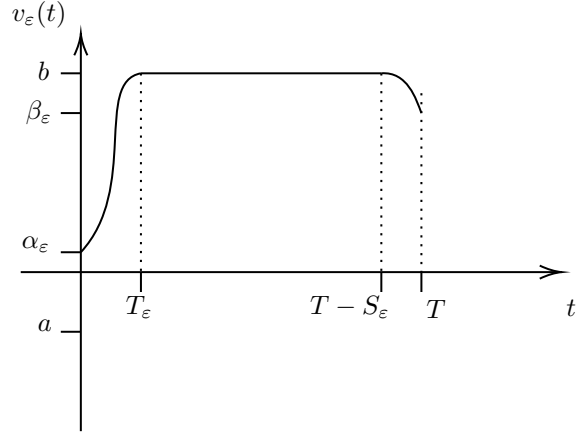
\begin{figure}
    \centering

\tikzset{every picture/.style={line width=0.75pt}} 

\begin{tikzpicture}[x=0.75pt,y=0.75pt,yscale=-1,xscale=1]

\draw    (190,200) -- (458,200) ;
\draw [shift={(460,200)}, rotate = 180] [color={rgb, 255:red, 0; green, 0; blue, 0 }  ][line width=0.75]    (10.93,-3.29) .. controls (6.95,-1.4) and (3.31,-0.3) .. (0,0) .. controls (3.31,0.3) and (6.95,1.4) .. (10.93,3.29)   ;
\draw    (220,280) -- (220,82) ;
\draw [shift={(220,80)}, rotate = 90] [color={rgb, 255:red, 0; green, 0; blue, 0 }  ][line width=0.75]    (10.93,-3.29) .. controls (6.95,-1.4) and (3.31,-0.3) .. (0,0) .. controls (3.31,0.3) and (6.95,1.4) .. (10.93,3.29)   ;
\draw    (220,230) -- (210,230) ;
\draw    (220,100) -- (210,100) ;
\draw    (210,190) -- (220,190) ;
\draw    (220,120) -- (210,120) ;
\draw    (220,190) .. controls (249,158.33) and (226.33,104.33) .. (250,100) ;
\draw    (250,200) -- (250,210) ;
\draw    (250,100) -- (370,100) ;
\draw    (370,200) -- (370,210) ;
\draw    (390,200) -- (390,210) ;
\draw    (370,100) .. controls (383,99) and (387.4,113) .. (390,120) ;
\draw  [dash pattern={on 0.84pt off 2.51pt}]  (370,100) -- (370,200) ;
\draw  [dash pattern={on 0.84pt off 2.51pt}]  (250,100) -- (250,200) ;
\draw  [dash pattern={on 0.84pt off 2.51pt}]  (390,110) -- (390,200) ;

\draw (197,222.4) node [anchor=north west][inner sep=0.75pt]    {$a$};
\draw (197,92.4) node [anchor=north west][inner sep=0.75pt]    {$b$};
\draw (189,180.4) node [anchor=north west][inner sep=0.75pt]    {$\alpha _{\varepsilon }$};
\draw (191,112.4) node [anchor=north west][inner sep=0.75pt]    {$\beta _{\varepsilon }$};
\draw (241,210.4) node [anchor=north west][inner sep=0.75pt]    {$T_{\varepsilon }$};
\draw (333,210.4) node [anchor=north west][inner sep=0.75pt]    {$T-S_{\varepsilon }$};
\draw (391,212.4) node [anchor=north west][inner sep=0.75pt]    {$T$};
\draw (461,212.4) node [anchor=north west][inner sep=0.75pt]    {$t$};
\draw (184,62.4) node [anchor=north west][inner sep=0.75pt]    {$v_{\varepsilon }( t)$};

\end{tikzpicture}

    \caption{Graph of the recovery sequence $v_\eps$ defined in equation (\ref{eq:def_recov_seq}).}
    \label{fig:plot_recovery}
\end{figure}
It follows that $v_{\eps}\in \h^1([0,T])$ with $v_{\eps}(0)=\alpha_{\eps}$, $v_{\eps}(T)=\beta_{\eps}$ and 
\begin{equation}
\label{eq:rec_seq_der}
    v'_{\eps}(t)=
   \begin{cases}
        \frac{\sqrt{W(v_{\eps}(t))}}{\eps}, & \textup{if $t\in [0,T-S_{\eps})$}, \\
         - \frac{\sqrt{W(v_{\eps}(t))}}{\eps} , & \textup{if $t\in [T-S_{\eps},T]$}. \\
   \end{cases} 
\end{equation}
By using equation (\ref{eq:rec_seq_der}) and substitution, we thus obtain the inequality
\begin{align*}
    G_\eps^{(0)} ( v_\eps ) 
    & \leq
    \eps \omega_1
    \left(
    \int_0^{T_\eps}
    \sqrt{W ( v_\eps ) } \abs{ v_\eps ' }
    \dd{ t }
    +
    \int_{T - S_\eps }^{T}
    \sqrt{W ( v_\eps ) } \abs{ v_\eps ' }
    \dd{ t }
    \right)
    \\
    & =
    \eps \omega_1 \left( 
    \dw( \alpha_\eps , b ) + \dw ( \beta_\eps , b )
    \right),
    \end{align*}
    which completes the proof.
\end{proof}
Recalling the definition of the higher-order expansion (see \Cref{sct:gamma_limit}) and noting that $ \min G^{(0)} = 0$, we define $G^{(1)}_{\eps}:\lp^1(I)\to [0,\infty]$ as 
\begin{equation*}
G^{(1)}_{\eps}(v)\coloneqq
\begin{cases}
    \int_I \left(\frac{W(v(t))}{\eps}+\eps(v'(t))^2\right)\omega(t) \dd{t} , & \textup{if $v\in \h^1(I)$, $v(0)=\alpha_{\eps}$ and $v(T)=\beta_{\eps}$}, \\
    \infty, & \textup{otherwise}.
\end{cases}
\end{equation*}

\begin{proposition}\label{p:first order gamma limit}
Let $G^{(1)}:\lp^1(I)\to [0,\infty]$ be the functional given by 
\begin{equation*}
G^{(1)}(v)=
\begin{cases}
    \frac{C_W}{b-a}\int_I \omega \dd{ \abs{ \mathrm{D} v } }+\dw(v(0),\alpha)\omega(0)+\dw(v(T),b)\omega(T), & \textup{if $v\in \bv(I;\{a,b\})$}, \\
    \infty, & \textup{otherwise}.
\end{cases}
\end{equation*}
Then
\begin{equation*}
    \Gamma \textup{-}\lim_{\eps \to 0} G^{(1)}_{\eps}=G^{(1)}.
\end{equation*}
\end{proposition}
\begin{proof}
    The proof follows as in \cite{ModicaMortola77} or in \cite{RubinsteinSternberg90}. For the lower bound, the additional boundary value terms are obtained by extending constant at the boundary and considering a slightly larger domain. For the recovery sequence, we make an additional transition at the boundary.
\end{proof}
\subsection{Recovery sequence for the second-order functional}

To define the second-order functional, we first have to find $ \min G^{(1)} $. We find that as long as the interval length is sufficiently small, the unique minimizer of $ G^{(1)} $ must already be the constant function $ b $.
\begin{proposition}
    Assume that (\ref{eq:bdry_assumptions_1d}) holds. Then the constant function $ b $ is the unique minimizer of $ G^{(1)}$.
\end{proposition}
\begin{proof}
    Since (\ref{eq:bdry_assumptions_1d}) distinguishes two cases, we first assume that $ \alpha_\eps \geq \alpha_-$ so that $ \alpha\geq \alpha_-$ as well.
    The energy of the constant function $ b $ is $ G^{(1)} ( b ) = \dw ( b , \alpha ) \omega (0)$.
    Let $ v \in \bv ( I , \{a,b\} ) $ be another competitor. We first assume that $ v (T)=a$. Then for $ \omega_1 - \omega_0 $ sufficiently small, we have
    \begin{equation*}
        G^{(1)} (v )
        \geq
        \dw ( a, b ) \omega ( T ) 
        =
        \dw ( a,  \alpha ) \omega ( T ) 
        +
        \dw ( \alpha, b )
        ( \omega ( T ) - \omega ( 0 ) )
        +
        \dw ( \alpha , b ) \omega (0 )
        > \dw ( \alpha , b ) \omega ( 0 ).
    \end{equation*}
    If $ v(  T ) = b $, then a non-constant competitor has to make at least one transition, which yields 
    $ G^{(1)} (v ) \geq \dw ( a, b ) \omega_0$, and the same computation as above gives that $G^{(1)} ( v ) > G^{(1)} ( b )$.

    We now treat the second case in (\ref{eq:bdry_assumptions_1d}), namely that $ \omega'(0) > - \kappa_0 $. The interval length has been chosen so small that inequality (\ref{eq:weight_lower_bound}) holds. If $ v $ is any competitor with $ v ( T ) =a $, then 
    \begin{equation}
    \label{eq:simple_comp}
        G^{(1)} ( v ) 
        \geq
        \dw ( a, b ) \omega ( T ) 
        >
        \dw ( \alpha, b ) \omega ( 0 ) 
        = G^{(1)} ( b ),
    \end{equation}
    thus $ v $ can not be a minimizer. If $ v ( T ) = b $, but $ v $ is not constant, then $ v $ has at least one transition and a computation as in (\ref{eq:simple_comp}) yields the desired result.
\end{proof}
As a consequence, we define for $\eps>0$ the functional $G^{(2)}_{\eps}:\lp^1(I)\to [0,\infty]$ as 
\begin{equation*}
G^{(2)}_{\eps}(v) \coloneqq
\begin{cases}
    \int_I \left(\frac{1}{\eps^2}W(v(t))+(v'(t))^2\right)\omega(t)\dd{t}-\frac{\omega(0)}{\eps}\dw(b,\alpha) ,  & \textup{if $v\in \h^1(I)$, $v(0)=\alpha_{\eps}$ and $v(T)=\beta_{\eps}$}, \\
    \infty,  & \textup{otherwise}.
\end{cases}
\end{equation*}
\begin{proposition}
Assume that 
\begin{equation}
    \label{eq:assumptions_1d_limsup}
  \dw( \alpha_\eps , \alpha ), \dw ( \beta_\eps , b ) \in o(\eps),
\end{equation}
and (\ref{eq:bdry_assumptions_1d}) holds.
Then there exists a sequence $(v_{\eps})_\eps$ such that $v_{\eps}\in \h^1(I)$ for all $ \eps > 0$, $v_{\eps}(0)=\alpha_{\eps}$ and $v_{\eps}(T)=\beta_{\eps}$ for every $\eps>0$ with $v_{\eps}\to b$ in $\lp^1(I)$ as $\eps \to 0$ and 
\begin{equation*}
    \limsup_{\eps \to 0}G^{(2)}_{\eps}(v_{\eps}) \leq 
    \omega'(0)\int_0^{\infty}2\sqrt{W(z_{\alpha}(s))}z_{\alpha}'(s)s \dd{s}
    =
    \omega' ( 0 ) \int_0^\infty
    \dw ( z_\alpha( s )  , b )
    \dd{ s }
    \eqqcolon G^{(2)} ( b ) 
\end{equation*}
\end{proposition}

\begin{proof}
For given $ \eps > 0 $, define $ v_\eps $ as in the proof of \Cref{prop:minimizer_properties}, see equation (\ref{eq:def_recov_seq}).
Since $\omega\in C^1([0,1])$, we have that 
\begin{equation*}
    \omega(t)=\omega(0)+\omega'(0)t+R(t) \quad \textup{for every $t\in [0,T]$}
\end{equation*}
where $R\in C^1([0,1])$ with $R(0)=R'(0)=0$. Write
\begin{align*}
    G^{(2)}_{\eps}(v_{\eps})
    ={}&
    \left(\int_0^{T-S_{\eps}}\frac{W(v_{\eps})}{\eps}+\eps(v'_{\eps})^2 \dd{t}-\dw(b,\alpha)\right)\frac{\omega(0)}{\eps} 
    \\ 
    &+ 
    \frac{\omega'(0)}{\eps}\int_0^{T-S_{\eps}}\left(\frac{W(v_{\eps})}{\eps}+\eps(v'_{\eps})^2\right)t \dd{t} 
    \\ 
    & + 
    \frac{1}{\eps}\int_0^{T-S_{\eps}}\left(\frac{W(v_{\eps})}{\eps}+\eps(v'_{\eps})^2\right)R(t) \dd{t} 
    \\
    & + 
    \frac{1}{\eps}\int_{T-S_{\eps}}^{T}\left(\frac{W(v_{\eps})}{\eps}+\eps(v'_{\eps})^2\right)\omega \dd{t} 
    \\
    \eqqcolon {}& \mathcal{A}_{\eps}+\mathcal{B}_{\eps}+\mathcal{C}_{\eps}+\mathcal{D}_{\eps}.
\end{align*}
We first estimate $\mathcal{A}_{\eps}$. 
By definition of $v_{\eps}$ and the change of variable $s=v_{\eps}(t)$ we have that 
\begin{equation*}
\int_0^{T-S_{\eps}}\frac{W(v_{\eps})}{\eps}+\eps(v'_{\eps})^2 \dd{t}
=
\int_0^{T-S_{\eps}}2\sqrt{W(v_{\eps})}v'_{\eps}\dd{t} 
= 
2\int_{\alpha_{\eps}}^{b}\sqrt{W(s)} \dd{ s},
\end{equation*}
thus
\begin{equation*}
\mathcal{A}_\eps 
\leq
\omega( 0 ) \frac{ \dw ( \alpha_\eps, \alpha ) }{ \eps},
\end{equation*}
which vanishes by assumption (\ref{eq:assumptions_1d_limsup}).

Next we consider $\mathcal{B}_{\eps}$. Arguing as in the step before, using the change of variable $s=t/\eps$, we have that 
\begin{equation*}
    \mathcal{B}_{\eps}=
    2\omega'(0) \int_0^{T-S_{\eps}}\frac{W(v_{\eps}(t))}{\eps} \frac{t}{\eps} \dd{t}
    = 
    2\omega'(0) \int_0^{\infty}W(z_{\alpha_\eps} ( s ) )s \dd{s}
\end{equation*}
for $ z_{\alpha_\eps} $ as in equation (\ref{eq:cauchy_problem_full}). We can easily verify that $z_{\alpha_\eps}\to z_{\alpha}$ uniformly on $[0,\infty)$ as $\eps\to 0$. Thanks to dominated convergence, we obtain that
\begin{equation*}
    \lim_{\eps\to 0}\mathcal{B}_{\eps}= 2\omega'(0) \int_0^{\infty}W(z_{\alpha}(s))s \dd{s}
    =
    2\omega'(0)  \int_0^{\infty}\sqrt{W(z_{\alpha})}z'_{\alpha}(s)s \dd{s}.
\end{equation*}
Using integration by party, we note that
\begin{equation*}
    2 \omega'(0)
    \int_0^\infty \sqrt{ W ( z_\alpha ) } z_\alpha ' ( s ) s 
    \dd{ s }
    =
    -\int_0^\infty \partial_s \dw ( z_\alpha ( s ) , b ) s 
    \dd{ s }
    =
    \int_0^\infty \dw ( z_\alpha ( s ) , b ) 
    \dd{s }.
\end{equation*}
Let us now consider the third term $\mathcal{C}_{\eps}$. First we observe that due to the differentiability of $ \omega $ and the inequality $ T_\eps \leq C \eps $ due to the local integrability of $ W^{-1/2}$, we have
\begin{equation*}
    \lim_{\eps\to 0} \frac{\max_{t\in [0,T_{\eps}]}|R(t)|}{\eps}=0.
\end{equation*}
In particular
\begin{equation*}
   \abs{\mathcal{C}_{\eps}}
  \leq 
  \frac{\max_{t\in [0,T_{\eps}]}|R(t)|}{\eps} 
  \int_0^{T_{\eps}}\frac{W(v_{\eps})}{\eps}+\eps(v'_{\eps})^2 \dd{t}  
  \leq
  \frac{\max_{t\in [0,T_{\eps}]}|R(t)|}{\eps}\dw ( \alpha_-, b )  
\end{equation*}
vanishes as $ \eps $ tends to zero.

Lastly we estimate $\mathcal{D}_{\eps}$. By definition of $v_{\eps}$ and 
\Cref{p:prop W} we have that 
\begin{align*}
    \mathcal{D}_{\eps}
    &= 
    \frac{1}{\eps}\int_{T-S_{\eps}}^{T}\left(\frac{W(v_{\eps})}{\eps}+\eps(v'_{\eps})^2\right)\omega \dd{t}
    \leq 
    \frac{2\omega_1}{\eps}\int_{T-S_{\eps}}^{T}\sqrt{W(v_{\eps})}v'_{\eps} \dd{t} 
    \\ 
    & =
    \frac{2\omega_1}{\eps}\int_{\beta_{\eps}}^b \sqrt{W(s)}\dd{s}
    =
    \frac{2\omega_1}{\eps} \dw(\beta_\eps, b ),
\end{align*}
which vanishes by equation (\ref{eq:assumptions_1d_limsup}).
\end{proof}
\subsection{Lower bound for the second-order functional}

In this section we derive quantitative lower bounds for $ v_\eps \in \argmin G_\eps^{(0)} = \argmin G_\eps^{(2)}$. We start by rewriting the Euler--Lagrange equation (\ref{e:E-L eq}) and note that it is similar to the Euler--Lagrange equation in the unweighted case up to a small error $ \delta_\eps$.

\begin{lemma}
\label{prop:euler_lagrange}
For $\eps>0$ and $v_{\eps}\in \argmin G_{\eps}^{(0)}$, we have that 
\begin{equation}
    \label{eq:euler_lagrange_with_delta}
    v_{\eps}'(t)^2=\frac{W(v_{\eps}(t))+\delta_{\eps}(t)}{\eps^2} \quad \quad \textup{for every $t\in [0,T]$},
\end{equation}
where 
\begin{equation*}
    \delta_{\eps}(t)\coloneqq\frac{1}{\omega(t)}\left( \omega(0)  \left( \eps^2 v_{\eps}'(0)^2-W(v_{\eps}(0)) \right) -\eps\int_0^t \left(\frac{1}{\eps}W(v_{\eps}(t))+\eps(v_{\eps}')^2(t)\right) \omega'(t)\dd{t}\right).
\end{equation*}
\end{lemma}

\begin{proof}
Define the function $\hat{\delta}_{\eps}:[0,T]\to \R$ as 
\begin{equation*}
    \hat{\delta}_{\eps}(t):=\eps^2 v_{\eps}'(t)^2-W(v_{\eps}(t)).
\end{equation*}
Multiplying equation (\ref{e:E-L eq}) by $2\eps^2v'(t)\omega(t)$ yields
\begin{equation*}
   \omega  (\eps^2 (v_{\eps}')^2-W(v_{\eps}))'  =-2\eps^2\omega'(v_{\eps}')^2
\end{equation*}
for every $t\in [0,T]$ and in particular
\begin{equation*}
    (\omega \hat{\delta}_{\eps})'(t) =-\omega'(t)(W(v_{\eps}(t))+\eps^2(v_{\eps}'(t))^2) 
\end{equation*}
for every $t\in [0,T]$. Therefore we obtain that 
\begin{equation*}
    \hat{\delta}_{\eps}(t)=\frac{1}{\omega(t)}\left(\omega ( 0 )   \left( \eps^2 v_{\eps}'(0)^2-W(v_{\eps}(0)) \right) -\eps\int_0^t \omega'(t)\left(\frac{1}{\eps}W(v_{\eps}(t))+\eps(v_{\eps}'(t))^2\right)\dd{t}\right)
\end{equation*}
for every $t\in [0,T]$, which is what we wanted to show.
\end{proof}

Next we derive an upper bound for the derivative of minimizers $ v_\eps \in \argmin G_\eps^{(0)}$, which has been shown in \cite[Cor.~3.9]{fonseca2025secondordergammalimitcahnhilliardfunctional1}.
\begin{lemma}\label{p:upper bound derivative}
There exists a constant $C_0>0$ only depending on $ W $, $ \norm{\omega}_{C^2} $ and $ T $ such that for every $\eps\in (0, T)$ and for every $v_\eps\in \argmin G_{\eps}^{(0)}$ we have that 
\begin{equation*}
   |v_\eps'(t)|\leq \frac{C_0}{\eps}
\end{equation*}
for every $t\in [0,T]$.
\end{lemma}

\begin{proof}
By the energy bound (\ref{eq:energy_bound}), there exists $C_0 > 0 $ such that
\begin{equation*}
    G_{\eps}^{(0)}(v_\eps)\leq \eps C_0.
\end{equation*}
Thanks to the continuity of $v_\eps'$ and $\omega$, we can find $t_0\in [0,T]$ such that 
\begin{equation*}
    \eps^2(v_\eps'(t_0))^2\omega(t_0)=\frac{1}{T}\int_0^T\eps^2(v_\eps'(t))^2 \omega(t) \dd{t} \leq \frac{\eps C_0}{T}
\end{equation*}
and thus
\begin{equation*}
    \eps |v_\eps'(t_0)|\leq \sqrt{\frac{\eps C_0}{T \omega_0}}.
\end{equation*}
By the Euler--Lagrange equation (\ref{e:E-L eq}) and the product rule $ ( v'\omega)'=v''\omega + v'\omega' $, we have that
\begin{equation*}
2\eps^2v_\eps'(t)\omega(t)=2\eps^2v_\eps'(t_0)\omega(t_0)+\int_{t_0}^tW'(v_\eps(s))\omega(s) \dd{s} 
\end{equation*}
for every $t\in [0,T]$. Since $a\leq v \leq b$, we deduce
\begin{equation}
\label{eq:estim_for_v'}
  2\eps^2|v_\eps'(t)|
  \leq 
  2\eps^{3/2} \sqrt{\frac{C_0}{T \omega_0^3}}\omega_1+ \max_{[a,b]}|W'|\int_0^T\omega \dd{s}  
  \leq 
  C_1 
\end{equation}
for every $t\in [0,T]$ and some suitable $ C_1 >0 $. 
Again by the Euler--Lagrange equation \eqref{e:E-L eq} and by inequality (\ref{eq:estim_for_v'}), we estimate for the second derivative
\begin{align*}
    \eps^2 |v_\eps''(t)|\leq C_2.
\end{align*}
Fix now $t\in [0,T]$ and let $t_1\in [0,T]$ be such that $|t-t_1|=\eps$. By the mean value theorem there exists $\theta$ between $t$ and $t_1$ such that 
\begin{equation*}
    v_\eps(t)-v_\eps(t_1)=v_\eps'(\theta)(t-t_1),
\end{equation*}
By the Fundamental Theorem of Calculus we therefore have that for every $t\in [0,T]$
\begin{equation*}
   |v_\eps'(t)|=\left |v_\eps'(\theta)+\int_{\theta}^t v_\eps''(s)\dd{s}\right|=\left |\frac{v_\eps(t)-v_\eps(t_1)}{t-t_1}+\int_{\theta}^t v_\eps''(s)\dd{s} \right|\leq \frac{(b-a)}{\eps}+ \frac{C_2}{2\eps}\leq \frac{C}{\eps},
\end{equation*}
for some $C>0$ that depends only on $W$, $\omega$ and $T$.
\end{proof}

\begin{lemma}  
\label{lem:conv_of_rescaled}
Let $v_{\eps} $ be a minimizer of $ G_{\eps}^{(0)}$ and define
\begin{equation*}
    w_{\eps}(t) \coloneqq
    \begin{cases}
        v_{\eps}(\eps t),  & \text{if } 0\leq t \leq T\eps^{-1}, \\
        \beta_{\eps},  & \text{if } T\eps^{-1} < t.
    \end{cases}
\end{equation*}
Then  $ w_{\eps}\to z_{\alpha}$ in $\w^{1,\infty}_{\textup{loc}}([0,\infty))$ as $ \eps $ tends to zero.
\end{lemma}

\begin{proof}
Thanks to Proposition \ref{p:upper bound derivative} we have for every $\eps>0$ sufficiently small that $w_{\eps}\in C^2_{\mathrm{loc}} ([0,\infty ) ) $ and
\begin{equation}\label{e:bound Linf of derivatives}
    |w_{\eps}'(t)|\leq C_0
\end{equation}
for all $ t \geq 0 $. 
In particular, for every $N\in \N$ there exists a constant $C_N>0$, such that 
\begin{equation*}
    \|w_{\eps}\|_{\h^1((0,N))}\leq C_N
\end{equation*}
for every $\eps\in (0,1)$. By a diagonalization argument, we find a non-relabelled subsequence and $w_0\in \h^1_{\textup{loc}}((0,\infty))$ such that $w_{\eps}\to w_0$ uniformly on compact subsets of $[0,\infty)$ and $w'_{\eps}\weakto w'_0$ in $L^2((0,N))$ for every $N\in \N$. 
Using the change of variable $s=t\eps^{-1}$, we see that
$w_{\eps}$ is a minimizer of the functional $H_{\eps}$ given by
\begin{equation*}
    H_{\eps}(w):=\int_0^{T\eps^{-1}}(W(w(s))+(w'(s))^2)\omega_{\eps}(s) \dd{s} 
\end{equation*}
where the weight is rescaled by $ \omega_{\eps}(s) \coloneqq \omega(\eps s)$ and subject to the Dirichlet boundary conditions $ w ( 0 ) = \alpha_\eps $ and $ w ( \eps^{-1} T ) = \beta_\eps$. 
Using outer variations we obtain that
\begin{equation*}
    \int_0^{T\eps^{-1}}(W'(w_{\eps})\phi+2w_{\eps}'\phi')\omega_{\eps} \dd{s}=0
\end{equation*}
for every $\phi\in C^{\infty}_c((0,T\eps^{-1}))$. 
Thus for every $N\in \N$ and every $ \Phi \in C_{c}^\infty ((0, N ) )$ we have
\begin{equation*}
    \int_0^N(W'(w_0)\phi+2w_0'\phi')\omega(0)\dd{s}=0.
\end{equation*}
By \Cref{prop:minimizer_properties}, $w_0$ is in $ C^2([0,\infty))$ and satisfies 
\begin{equation*}
    \begin{cases}
        2w_0''(s)=W'(w_0(s)) \; & \textup{for every $s\in [0,\infty)$}\\
        w_0(0)=\alpha.
    \end{cases}
\end{equation*}
Multiplying by $w_0'(s)$ and integrating, we obtain that there exists $c\in \R$ such that
\begin{equation}\label{e:E-L 2 of w_0}
    (w_0'(s))^2-W(w_0(s))=c
\end{equation}
for every $s\in [0,\infty)$. 
We note that we already have $ c = 0 $ since $ W $ is integrable and $ w_0' $ square-integrable. In fact, we note by the uniform convergence and weak convergence of the derivatives on compact sets, we have due to the dominated and monotone convergence theorem that
\begin{align*}
    \omega_0 \int_0^\infty W ( w_0 ) + (w_0')^2 \dd{ s }
    & =
    \lim_{N\to \infty }
    \omega_0 \int_0^N W ( w_0) + (w_0')^2 \dd{ s }
    \\
    & \leq 
    \liminf_{N \to \infty }
    \liminf_{\eps \to 0}
    \omega_0 
    \int_0^N W ( w_0 ) + (w_\eps ')^2 \dd{ s }
    \\ 
    & \leq
    \liminf_{\eps \to 0}
    \int_0^{ \eps^{-1} T }
    ( W ( w_\eps ) + (w_\eps ')^2  )\omega_\eps
    \dd{ s }
    \\
    & =
    \liminf_{\eps \to 0 }
    \frac{1}{\eps}
    \int_0^T 
    \left( 
    W ( v_\eps )
    +
    \eps^2 (v_\eps ' )^2
    \right) 
    \omega \dd{ s },
\end{align*}
which is finite by \Cref{prop:minimizer_properties}.
Moreover, thanks to \eqref{e:bound Linf of derivatives} and the Euler--Lagrange equation \eqref{e:E-L eq}, there exists $C'\in (0,\infty)$, depending only on $W$, $\norm{\omega}_{ C^2 } $ and $T$, such that for fixed $N\in \N$, for every $\eps>0$ small enough we have that $w_{\eps}\in C^2([0,N])$ and
\begin{equation*}
    |w_{\eps}'(s)|+|w_{\eps}''(s)|\leq C'
\end{equation*}
for every $s\in [0,N]$. In particular, by the Arzelà--Ascoli Theorem, we can conclude by equation (\ref{e:E-L 2 of w_0}) that $w_{\eps}\to z_{\alpha}$ in $\mathrm{W}^{1,\infty}_{\textup{loc}}((0,\infty))$, up to a subsequence, as $\eps \to 0$.
\end{proof}
In the case $ \alpha_\eps \geq \alpha_-$, it is energetically not favourable for minimizers to make a transition down to $ a $. Even though we could save energy by doing the transition to $ \beta_\eps $ close to $ T $ since the weight could in principle be decreasing, we would have to transition the area $ [a, \alpha_-] $ twice. But since the weight is close to being constant, this can not happen. We now make this argument rigorous.
\begin{lemma}
\label{prop:lower_bound}
Assume that $ \alpha_\eps, \beta_\eps \in [ \alpha_-, b ] $ for all $ \eps > 0 $ and let $ v_\eps \in \argmin G_\eps^{(0)}$.
Fix any $ m^\ast \in (a, \alpha_-)$. Then there exists $ \eps_0 > 0 $ and $ \delta> 0 $ only depending on $ m^\ast$ and $ W $ such that if $ \omega_1 - \omega_0 < \delta $, then
\begin{equation*}
    \min_{[0,T]}v_{\eps}\geq m^*
    \quad
    \text{for all }
    0 < \eps < \eps_0.
\end{equation*}
\end{lemma}

\begin{proof}
Fix $\eps> 0$ small, let $m_{\eps}\coloneqq\min_{[0,T]}v_{\eps}$ and $t_\eps\in [0,T]$ be such that $v_\eps(t_\eps)=m_\eps$. By Young's inequality and substitution, we get
\begin{align*}
   \frac{G_{\eps}^{(0)}(v_{\eps})}{\eps} & \geq 
   2\omega_0\int_0^{t_\eps}\sqrt{W ( v_\eps ) } \abs{v_\eps ' }\dd{s}
   +
   2\omega_0\int_{t_\eps}^T\sqrt{W ( v_\eps ) } \abs{v_\eps' }\dd{s} 
   \\ 
   & \geq
   \omega_0\int^{\alpha_{\eps}}_{m_\eps} 2\sqrt{W(s)} \dd{s}+\omega_0\int^{\beta_{\eps}}_{m_\eps} 2\sqrt{W(s)} \dd{s} 
   \\ 
   & = 
   2\omega_0\int^{\alpha_{\eps}}_{m_\eps} 2\sqrt{W(s)} \dd{s}+\omega_0\int^{\beta_{\eps}}_{\alpha_{\eps}} 2\sqrt{W(s)} \dd{s}.
\end{align*}
Since $ v_\eps $ is a minimizer, we therefore obtain by inequality (\ref{eq:energy_bound}) that
\begin{align*}
   \omega_0 \left( 
   2 \int_{ m_\eps }^{ \alpha_\eps} 
   2 \sqrt{W ( s )} \dd{ s }
   +
   \int_{ \alpha_\eps }^{\beta_\eps }
   2 \sqrt{W ( s )} \dd{ s }
   \right)
   & \leq
   \omega_1 \left( \dw ( \alpha_\eps , b ) + \dw ( \beta_\eps , b ) \right)
   \\
   & =
   \omega_1 \left( 
   2 \int_{ \beta_\eps }^b 2\sqrt{ W ( s )} \dd{ s }
   +
   \int_{ \alpha_\eps }^{ \beta_\eps } 
   2 \sqrt{W ( s )}
   \dd{ s }
   \right).
\end{align*}
In the limit $ \eps \to 0$, this inequality implies
\begin{align*}
    \limsup_{ \eps \to 0 } 
    2 \omega_0 \int_{m_\eps }^{ \alpha_-}
    2 \sqrt{W(s)}
    \dd{ s }
    \leq
    \limsup_{\eps \to 0}
    2\omega_0
     \int_{ m_\eps }^{\alpha_\eps }
    2 \sqrt{W ( s )} 
    \dd{ s }
    &\leq
    (\omega_1 - \omega_0 )
    \int_\alpha^{\beta}
    2 \sqrt{W ( s ) }\dd{ s }
    \\
    & \leq
    ( \omega_1 - \omega_0 ) 
    \int_{\alpha_-}^\beta 
    2\sqrt{W(s)}\dd{ s }.
\end{align*}
The right hand side gets arbitrarily small if $ \omega_1 - \omega_0 $ is small, and since $ m^\ast < \alpha_- $ is fixed, the desired claim follows.
\end{proof}
Since we have just shown that in the case $ \alpha_\eps \geq \alpha_-$ we do not transition down to the well $ a $, a minimizer $ v_\eps $ has to reach the well $ b $ in a time of scale $ \eps $, which is what we now show.
\begin{lemma}
    \label{lem:properties_of_1d_minimizers}
    There exists a sequence $ \theta_\eps \uparrow b $ with the following properties. 
    Assume that $ \alpha_\eps , \beta_\eps \in [\alpha_-, b ] $  for all $ \eps > 0 $ and that $ \dw ( \alpha_\eps, \alpha ), \dw ( \beta_\eps , b ) \in o ( \eps ) $. We have that $ \dw ( \theta_\eps , b ) \in o ( \eps )  $ and if $ v_\eps \in \arg \min G_\eps^{(0)} $ then there exists a unique smallest $ S_\eps \geq 0 $ such that $ v_\eps ( S_\eps ) = \theta_\eps $ and such that $ v_\eps $ is increasing on $ [0, S_\eps ] $. Moreover we have
    \begin{equation}
    \label{eq:Teps_asymptotics}
        \lim_{ \eps \to 0 }
        \frac{S_\eps}{\eps } = T^{ (\alpha) },
    \end{equation}
    where $ T^{(\alpha ) } $ is as in \Cref{prop:cauchy_problem}.
    If we quantify the estimates, then we get that for $ \eps>0 $ sufficiently small, there exists $M>0$ such that 
    \begin{equation}
        \label{eq:Seps_upper_bound}
        \frac{S_\eps}{\eps}
        \leq
        M,
    \end{equation}
    where $ M $ only depends on $ W $ and the convergence rate of $ \beta_\eps $ to $ b $, and there exists a constant $ C > 0 $ such that $ \theta_\eps $ satisfies the bound
    \begin{equation}
        \label{eq:thetaeps_bound}
        \theta_\eps \geq b - C \max \{ b - \beta_\eps , \eps^{1/(1+q)} \}.
    \end{equation}
\end{lemma}

\begin{proof}
    Without loss of generality, we may assume that $ \alpha < b $, since otherwise we can simply choose $ \theta_\eps = \alpha_\eps $ and immediately get the desired result.
    For $ C > 0 $, let $ \theta_\eps $ be defined by
    \begin{equation*}
        \theta_\eps \coloneqq b - C \max \{ b - \beta_\eps , \eps^{1/(1+q)} \}.
    \end{equation*}
    Let $ C' > 0 $ be chosen below. Then for $ C > 0 $ sufficiently large, we note that due to the growth bounds of $ W $ near $ b $, see \Cref{p:prop W}, we have that
    \begin{equation*}
        C' ( W ( \beta_\eps ) + \eps ) 
        \leq
        \frac{\sigma C^{1+q}}{2} ( W ( \beta_\eps ) + \eps )
        \leq
        W ( \theta_\eps )
        \quad
        \text{and}
        \quad
        \dw( \theta_\eps , b ) \in o( \eps ).
    \end{equation*} 
    
    Suppose now an $S_\eps $ as stated in the claim does not exist. Then we define
    \begin{equation*}
        t_\ast \coloneqq \arg \min \left\{ t \in [0, T] \colon v_\eps ' ( t ) = 0 \right\}.
    \end{equation*}
    By assumption we have $ t_\ast < T $, $ x_\ast \coloneqq v_\eps  ( t_\ast ) < \theta_\eps $ and $ v_\eps ' ( t_\ast ) = 0 $. Since we can piece together minimizers on disjoint intervals, $ v_\eps $ must also be a minimizer of the localized functional
    \begin{equation*}
        G_\eps^{[t_\ast, T]} ( v )
        \coloneqq 
        \int_{t_\ast}^T \left( \eps W ( v ) + \frac{1}{\eps } ( v'( t ) )^2 \right) \omega ( t )\dd{ t }
    \end{equation*}
    subject to the boundary conditions $ v ( t_\ast ) = v_\eps ( t_ \ast ) $ and $ v ( T ) = \beta_\eps $. By \Cref{prop:euler_lagrange}, we have that on $ [t_\ast , T ] $, the function $ v_\eps $ satisfies the Euler--Lagrange equation
    \begin{equation}
    \label{eq:el_on_tast_T}
        (v_\eps' ( t) )^2
        =
        \frac{ W ( v_\eps ( t ) ) + \delta_{\eps }^\ast ( t ) }{ \varepsilon^2}
    \end{equation}
    for 
    \begin{align*}
        \delta_{\eps }^\ast( t ) 
        &=
        \frac{ 1 }{\omega ( t ) }
        \left( 
        - \omega ( t_\ast ) W ( x_\ast ) 
        - \varepsilon \int_{t_\ast }^{ t }  \left( \frac{ 1}{ \varepsilon } W (v_\eps ( t ) ) + \varepsilon (v_\eps ( t ) )^2 \right) \omega' ( t ) \dd{ t }
        \right)
        \\
        & \eqqcolon
        \frac{1}{\omega ( t ) } 
        \left( 
        - \omega ( t_\ast ) W ( x_\ast ) - \varepsilon \Phi_{\eps}^\ast ( t ) \right).
    \end{align*}
    We note that by the energy bound in \Cref{prop:minimizer_properties}, the function $ \Phi_{\eps }^\ast $ stays uniformly bounded in the supremum norm. Plugging $ t = T $ into equation (\ref{eq:el_on_tast_T}) yields due to the choice of $ \theta_\eps $ that for $ C'> 0 $ sufficiently large, we have
    \begin{equation}
    \label{eq:contradiction_inequality}
        W (x_\ast ) \leq \frac{ 1 }{\omega ( t_\ast ) } \left( \omega( T ) W ( \beta_\eps ) - \eps \Phi_{\eps }^\ast ( T ) \right)
        \leq 
        C'( W ( \beta_\eps ) + \varepsilon )
        \leq 
        W ( \theta_\eps ).
    \end{equation}
    But now we have (for $ \eps>0 $ sufficiently small) a contradiction to $ a < x_\ast < \theta_\eps $ and \Cref{prop:lower_bound}, which proves our first claim.

    We now want to show the convergence claimed in equation (\ref{eq:Teps_asymptotics}). First we want to show the lower bound
    \begin{equation*}
        \liminf_{\eps \to 0 }
        \frac{ S_\eps }{ \eps }
        \geq
        T^{(\alpha)}.
    \end{equation*}
    If the left-hand side is infinity, we are done, so assume otherwise. We take a non-relabelled subsequence  such that the limit inferior is realized and such that there exists $ \Bar{T}$ so that for this subsequence, $ S_\eps / \eps \to \Bar{T}$. By the uniform convergence on compact intervals shown in \Cref{lem:conv_of_rescaled}, we must then have
    \begin{equation*}
        v_\eps ( S_\eps ) - z_\alpha ( S_\eps / \eps ) \to 0.
    \end{equation*}
    But $ v_\eps ( S_\eps ) = \beta_\eps \to b $ and $ z_\alpha ( S_\eps / \eps ) \to z_\alpha ( \Bar{T} ) $, from which we deduce that $ z_\alpha ( \Bar{ T } ) = b $. By definition of $ z_\alpha $, this yields $ \Bar{ T } \geq T^{(\alpha)} $, which is what we wanted to show.

    We will now show the reverse inequality
    \begin{equation}
    \label{eq:limsup_inequality_Teps}
        \limsup_{\eps \to 0 }
        \frac{ S_\eps }{ \eps }
        \leq
        T_\alpha.
    \end{equation}
    We do this by first showing that for $ \delta_\eps ( t ) $ as in \Cref{prop:euler_lagrange}, it holds that
    \begin{equation}
    \label{eq:delta_eps_decay}
        \frac{\min \{\delta_\eps ( t ), 0 \} }{ \eps^{2(1+q)/(3+q)}}
        \to 0
    \end{equation}
    uniformly on $[0,T]$.
    To see this, we plug $ t = T $ into the Euler--Lagrange equation (\ref{eq:euler_lagrange_with_delta}) to obtain
    \begin{equation}
        \label{eq:delta_eps_lower_bound_T}
        \delta_\eps ( T ) \geq - W ( \beta_\eps ).
    \end{equation}
    Since $ \dw (\beta_\eps, b ) \in o( \eps )$, it follows by \Cref{p:prop W} that 
    \begin{equation}
    \label{eq:W_beta_eps_decay}
    \frac{W( \beta_\eps ) }{ \eps^{2 ( 1+q)/(3+q)} }
    \to 0.
    \end{equation}
    Moreover we can write $ \delta_\eps $ as 
    \begin{equation*}
        \delta_\eps ( t ) = \frac{ 1 }{ \omega ( t ) } \left( C_\eps + \eps \Phi_\eps ( t ) \right)
    \end{equation*}
    for a constant $ C_\eps \in \R $ and a function $ \Phi_\eps $ which is bounded in $ \mathrm{L}^\infty ([0,T]) $ uniformly as $ \eps \to 0 $.
    This shows that $ \delta_\eps $ is Lipschitz continuous with constant less or equal to $ C \eps $.
    Thus inequality (\ref{eq:delta_eps_lower_bound_T}) combined with (\ref{eq:W_beta_eps_decay}) already implies the desired claim (\ref{eq:delta_eps_decay}) since $ 2 (1+q)/(3+q) <1$.

    To deduced the desired inequality (\ref{eq:limsup_inequality_Teps}), we use a comparison principle to the unweighted case, in which the claim follows easily via the explicit representation of solutions to the Cauchy problem.
    Define the lower bound $ \tilde{\delta}_\eps \coloneqq \min \{ \min_{[0,T] } \delta_\eps (t), 0 \} $ and define $ \tilde{\beta}_\eps \coloneqq W^{-1} ( -\tilde{\delta}_\eps ) $, where the inverse is taken on the interval $ [b - \mu, b ]$ for a sufficiently small $ \mu > 0 $ as in \Cref{p:prop W}. Moreover we define  the function
    \begin{equation}
    \label{eq:varphi_eps_def}
        \varphi_\eps ( x ) 
        \coloneqq
        \int_{ \alpha_\eps }^{ x }
        \frac{ \eps }{ \sqrt{ W ( \rho ) + \tilde{\delta}_\eps } }
        \dd{\rho }
    \end{equation}
    for $ x \leq \tilde{T}_\eps \coloneqq \varphi_\eps ( \tilde{\beta}_\eps )$. 
    Note that by definition of $ \tilde{\beta}_\eps$,  the function $ \varphi_\eps $ is well-defined on $ [\alpha_\eps , \tilde{\beta}_\eps ] $, see also inequality (\ref{eq:Winv_integral_ex}). Denote its inverse by $ u_\eps $, which is a solution to
    \begin{equation*}
    \begin{cases}
        u_\eps ' ( t ) = \frac{ \sqrt{ W ( u_\eps ) + \tilde{\delta}_\eps } }{ \eps }, & \textup{for } t \in [0, \tilde{T}_\eps],
        \\
        u_\eps ( 0 ) = \alpha_\eps,
        \\
        u_\eps ( \tilde{T}_\eps ) = \tilde{\beta}_\eps.
    \end{cases}
    \end{equation*}
    Using the Euler--Lagrange equation (\ref{eq:euler_lagrange_with_delta}), the definition of $ \tilde{\delta}_\eps $ and a comparison principle, it follows that $ v_\eps ( t ) \geq u_\eps ( t ) $ for all $ t \in [0, \tilde{T }_\eps ]$. 
    By the $ 1/(1+q)$-Hölder continuity of $ W^{-1}$ at $ b $ and \Cref{p:prop W} we have
    \begin{equation*}
        \frac{\dw ( \tilde{\beta_\eps}, b ) }{ \eps }
        \lesssim 
        \frac{(b- \tilde{\beta}_\eps )^{(3+q)/2}}{\eps}
        \lesssim 
        \frac{(-\tilde{\delta}_\eps)^{(3+q)/(2(1+q))}}{\eps}
        \to 0.
    \end{equation*}
    Thus by redefining $ \tilde{\theta}_\eps \coloneqq \min\{ \tilde{\beta}_\eps, \theta_\eps \} $, we may without loss of generality assume $ \theta_\eps \leq \tilde{ \beta }_\eps $ while keeping the properties of $ \theta_\eps $. Combined with $ v_\eps \geq u_\eps $, this yields
    \begin{equation*}
        v_\eps ( \tilde{T}_\eps ) 
        \geq 
        u_\eps ( \tilde{ T}_\eps ) 
        = 
        \tilde{\beta}_\eps 
        \geq 
        \theta_\eps.
    \end{equation*}
    Therefore by definition of $ S_\eps$, we have $ S_\eps \leq \tilde{T}_\eps $. We lastly check the asymptotics of $ \tilde{T}_\eps$. We claim that as $ \eps \to 0 $, we have 
    \begin{equation*}
        \frac{\tilde{T}_\eps}{\eps}
        =
        \int_{\alpha_\eps}^{\tilde{\beta}_\eps}
        \frac{ 1 }{ \sqrt{ W ( \rho ) + \tilde{\delta}_\eps }} 
        \dd{ \rho }
        \to 
        \int_{ \alpha}^{b} \frac{ 1 }{ \sqrt{ W ( \rho ) } } \dd{ \rho }
        =
        T^{(\alpha)}.
    \end{equation*}
    This follows by the generalized dominated convergence theorem. In fact, we note by \Cref{p:prop W} and equation (\ref{e:limite subquadratico}) that $ W^{-1}$ is $ 1/(1+q)$-Hölder continuous near $ b $. Thus for $ \mu > 0 $ sufficiently small we have that
    \begin{equation}
    \label{eq:Winv_integral_ex}
        \int_{ b - \mu }^{ \tilde{\beta}_\eps }
        \frac{1}{\sqrt{ W ( \rho ) + \tilde{\delta}_\eps } }
        \dd{ \rho }
        =
        \int_{ b - \mu }^{ \tilde{\beta}_\eps }
        \frac{1}{\sqrt{ W ( \rho ) - W ( \tilde{\beta}_\eps) } }
        \dd{ \rho }
        \lesssim
        \int_{ b - \mu }^{ \tilde{\beta}_\eps }
        (\tilde{\beta}_\eps - \rho)^{ - \frac{1+q}{2 } }
        \dd{ \rho }
        \leq
        \int_{0}^{ \mu }
        t^{-\frac{1+q}{2  } }
        \dd{ t },
    \end{equation}
    which is finite because $ q \in (0,1).$ 
    Thus we conclude
    \begin{equation*}
        \limsup_{ \eps \to 0 }
        \frac{ S_\eps }{\eps}
        \leq
        \limsup_{ \eps \to 0 }
        \frac{ \tilde{T}_\eps}{\eps }
        = 
        T^{(\alpha)}
    \end{equation*}
    For the upper bound (\ref{eq:Seps_upper_bound}) we first note that $ \delta_\eps ( t ) \geq - W( \beta_\eps ) - C \eps $ by the previous estimates. Combined with $ S_\eps \leq \tilde{T}_\eps $ and inequality (\ref{eq:Winv_integral_ex}), this yields the desired upper bound depending on the convergence speed of $ \beta_\eps $ to $ b $.
\end{proof} 
It remains to analyze the behaviour in the case $ \alpha_\eps \in [a, \alpha_-] $ as well, where we assume that $ \omega' \geq - \kappa_0 /2$. We first prove an analogue of \Cref{prop:lower_bound}. Due to the bound on the derivative of the weight, a transition up to $ \beta_\eps $ at the interval bound $ T $ would cost more energy then a transition at time zero. Thus a minimizer $ v_\eps \in \argmin G_\eps^{(0)} $ should immediately make a transition up to an area close to $ b $. This is the content of the following two results.
\begin{lemma}
\label{lem:bound_min_case2}
    If the weight $ \omega $ satisfies inequality (\ref{eq:weight_lower_bound}), then there exists some constant $ N > 0 $ only depending on $ W $, $ \norm{\omega}_{C^2} $, $ \kappa_0 $ and the convergence rate of $ \beta_\eps \to b $ such that for $ \eps > 0 $ sufficiently small, every $ v_\eps \in \argmin G_\eps^{(0)} $ satisfies
    \begin{equation*}
        \dw ( \alpha_\eps , \min v_\eps ) \leq N \eps.
    \end{equation*}
\end{lemma}
\begin{proof}
    Let $ t_\eps \in [0,T]$ be such that $ v_\eps ( t_\eps ) = m_\eps \coloneqq \min_{[0,T]} v_\eps $. As in the proof of \Cref{prop:lower_bound} and using inequality (\ref{eq:weight_lower_bound}), we then deduce that  
    \begin{align}
        \notag
        \frac{G_\eps^{(0)} ( v_\eps ) }{\eps}
        & \geq
        \left(\min_{[0,t_\eps]} \omega\right) \dw ( \alpha_\eps, m_\eps ) 
        +
        \left(\min_{[t_\eps, T ]} \omega\right) \dw ( m_\eps, \beta_\eps )
        \\
        \notag
        & \geq
        \omega ( 0 ) \dw ( \alpha_\eps, m_\eps )
        +
        \left(\omega (0 ) - \frac{\kappa_0}{2} t_\eps \right) \dw ( m_\eps , \beta_\eps )
        \\
        \label{eq:energy_of_min}
        & \geq
        2 \omega (0 ) \dw ( \alpha_\eps , m_\eps ) 
        +
        \left( \omega ( 0 )- \frac{\kappa_0}{2} t_\eps \right) \dw ( \alpha_\eps , \beta_\eps ).
    \end{align}
    We compare this energy with the optimal transition profile constructed in the proof \Cref{prop:minimizer_properties}, which has at most energy
    \begin{align}
    \label{eq:energy_upper_bound}
    &
        \left( \omega ( 0 ) - \frac{\kappa_0}{2 } C\eps \right) \dw ( \alpha_\eps, b ) 
        +
        \left(\max_{[0,T]} \omega\right) \dw ( \beta_\eps , b ) 
        \\
        ={} &
        \left( \omega ( 0 ) - \frac{\kappa_0}{2} C \eps \right) 
        \left( \dw ( \alpha_\eps , \beta_\eps ) + \dw ( \beta_\eps, b ) \right) 
        +
        \left(\max_{[0,T]} \omega \right) \dw ( \beta_\eps , b ).
    \end{align}
    Since $ v_\eps $ is a minimizer, it thus follows that we must have that the term (\ref{eq:energy_of_min}) is less or equal than (\ref{eq:energy_upper_bound}), from which we deduce that
    \begin{align*}
        2 \omega ( 0 ) \dw ( \alpha_\eps, m_\eps)
        & \leq
        - \frac{\kappa_0}{2} C\eps \dw ( \alpha_\eps , \beta_\eps ) 
        +
        \left( \omega(0) - \frac{\kappa_0}{2 }C \eps + \max \omega \right) \dw ( \beta_\eps, b ).
    \end{align*}
    Since $ \dw (\beta_\eps ,b ) \in o(\eps ) $, the right-hand side is, for $ \eps > 0 $ sufficiently small and dependent on the convergence rate of $ \beta_\eps $ to $ b $, bounded by $ N \eps/(2 \omega ( 0 ))$ for some sufficiently large $ N > 0 $, which proves the claim.
\end{proof}
\begin{lemma}
\label{lem:properties_1d_mins_close_to_a}
    There exists a sequence $ \theta_\eps \uparrow b $ with the following properties.
    Assume that $ \omega' ( 0 ) \geq - \kappa_0 $ and $ \dw ( \alpha_\eps , \alpha ) , \dw ( \beta_\eps, b ) \in o ( \eps ) $. We have that $ \dw( \theta_\eps , b ) \in o ( \eps ) $ and if $ v_\eps \in \argmin G_\eps^{(0)} $, then there exists $ S_\eps \geq 0 $ such that $ v_\eps ( S_\eps ) = \theta_\eps $ and there exists $ M > 0 $ such that for $ \eps > 0 $ sufficiently small, 
    \begin{equation}
    \label{eq:Seps_upper_bound_close_a}
        \frac{S_\eps}{\eps }
        \leq 
        M,
    \end{equation}
    where $ M $ only depends on $ \alpha_-$, $ \kappa_0 $, $ W $ and the convergence rate of $ \beta_\eps $ to $ b $. Moreover there exists a constant $ C > 0 $ such that for $ \eps > 0 $ sufficiently small, $ \theta_\eps $ satisfies the bound
    \begin{equation}
    \label{eq:thetaeps_bound_close_a}
        \theta_\eps 
        \geq
        b 
        - 
        C \max \{ b - \beta_\eps , \eps^{1/(1+q)} \}.
    \end{equation}
\end{lemma}
\begin{proof}
    The idea of the proof is to argue that we find a small time $ t_\eps $ in which we reach $ v_\eps ( t_\eps ) = m_\eps > a  $ and then apply \Cref{lem:properties_of_1d_minimizers} on the smaller interval $ [t_\eps , T] $.

    We first argue that there exists some constant $ D \geq 0 $ depending only on $ \kappa_0 $, $ W $ and the convergence rate of $ \beta_\eps $ to $ b $ such that for $ \eps > 0 $ sufficiently small, we have $ \dw( v_\eps ( D \eps ) , a ) \geq 2 N \eps  $ for the constant $ N $ as in \Cref{lem:bound_min_case2}. For this we estimate
    \begin{equation*}
        \frac{G_\eps^{(0)} ( v_\eps ) }{ \eps }
        \geq
        \left(\min_{[D \eps, T ]} \omega\right) \dw ( v_\eps ( D\eps ) , \beta_\eps ) 
        \geq
        \left( \omega ( 0 ) - \frac{\kappa_0}{2} D \eps \right)
        \left( \dw ( \alpha_\eps , \beta_\eps ) - \dw ( v_\eps ( D\eps ), \alpha_\eps ) \right).
    \end{equation*}
    Again using the energy bound from \Cref{prop:minimizer_properties} (as in (\ref{eq:energy_upper_bound})), we deduce that
    \begin{equation}
    \label{eq:distance_bound}
        \left( \omega ( 0 ) - \frac{\kappa_0}{2} D \eps \right) \dw ( v_\eps ( D\eps ) ,\alpha_\eps )
        \geq 
        - \frac{\kappa_0}{2} ( D-C ) \eps  \dw ( \alpha_\eps, \beta_\eps ) + C \dw ( \beta_\eps , b ).
    \end{equation}
    If $ \alpha_\eps \geq \alpha_-$, then we already know from \Cref{prop:lower_bound} that $ \dw ( v_\eps (D \eps ), a ) \geq N \eps $, so we may assume $ \alpha_\eps < \alpha_-$. We then note that if $ \eps> 0 $ is sufficiently small, we can due to inequality (\ref{eq:distance_bound}) choose $ D $ sufficiently large such that $\dw ( v_\eps ( D \eps ) , \alpha_\eps ) \geq 5 N \eps $. Using \Cref{lem:bound_min_case2}, we can use the triangle inequality and $ v_\eps \in [a,b]$ to deduce
    \begin{align*}
        2 \dw ( a , v_\eps ( D_\eps ) )
        & \geq
        \dw ( \alpha_\eps , v_\eps ( D_\eps ) ) - \dw ( \alpha_\eps , a ) + \dw ( a , v_\eps ( D \eps ) )
        \\
        & \geq 
        5 N \eps - \dw ( \alpha_\eps , \min v_\eps ) - \dw ( a , \min v_\eps ) + \dw ( a , \min v_\eps ) 
        \geq
        4 N \eps,
    \end{align*}
    which shows the desired claim.

    Next, we apply the strategy of \Cref{lem:properties_of_1d_minimizers} on the interval $ [D\eps , T ]$. To be precise we claim that the shifted function $ \bar{v}_\eps ( t )  \coloneqq v_\eps ( t +D \eps )  $ satisfies $ \bar{v}_\eps ( \bar{S}_\eps ) = \theta_\eps $ for $ \theta_\eps $ as claimed in the statement of \Cref{lem:properties_1d_mins_close_to_a} and $ \bar{S}_\eps $ as in the statement of \Cref{lem:properties_of_1d_minimizers}. We prove the claim by following the proof steps of \Cref{lem:properties_of_1d_minimizers}, for which not all assumptions are satisfied. To be precise the first difficulty arises when trying to deduce a contradiction from inequality (\ref{eq:contradiction_inequality}), which yielded $ W ( x_\ast ) \leq W ( \theta_\eps ) $.
    The problem is that we can not apply \Cref{prop:lower_bound} here. 
    However we know by \Cref{lem:bound_min_case2} and the first claim of this proof that
    \begin{equation*}
        \dw ( a , x_\ast ) \geq \dw ( a , \min \bar{v}_\eps ) \geq 
        \dw ( a , \bar{v}_\eps ( 0 ) ) - \dw ( \bar{v}_\eps ( 0 ) , \min \bar{v}_\eps )
        \geq
        \dw ( a, \bar{v}_\eps ( 0 ) ) - N\eps
        \geq N \eps.
    \end{equation*}
    From $ W ( x_\ast ) \leq W ( \theta ) $ and $ x_\ast < \theta_\eps $, we deduce due to the monotonicity of $ W $ at $ b $ that $ x_\ast $ has to be close to $ a $. Thus $ \dw ( a, x_\ast ) \geq N\eps $ combined with $ \dw ( \theta_\eps , b ) \in o ( \eps ) $ is already a contradiction to $ W ( x_\ast ) \leq W ( \theta ) $ for $ \eps > 0 $ sufficiently small.
    This proves that $ \bar{v}_\eps ( \bar{S}_\eps ) = \theta_\eps $ as in the proof of \Cref{lem:properties_of_1d_minimizers}. 
    Moreover $ \bar{S}_\eps $ satisfies the upper bound $ \bar{S}_\eps / \eps \leq M $ by using the same proof strategy as in \Cref{lem:properties_of_1d_minimizers}. In fact, we note by the first claim of this proof that 
    \begin{equation*} 
    W ( \bar{v}_\eps ( 0 ) ) \gtrsim ( a - \bar{v}_\eps ( 0 ))^{1+q} \gtrsim \dw ( a , \bar{v}_\eps (0) )^{2 (1+q)/(3+q) } \gtrsim \eps^{2 ( 1+q )/(3+q) }.
    \end{equation*}
    By inequality (\ref{eq:delta_eps_decay}), we have $ \delta_\eps \in o ( \eps ^{(3+q)/(2(1+q))} ) $, so that for $ \eps > 0 $ sufficiently small, we have $ W(\bar{v}_\eps ( 0 )) \geq \delta_\eps $, which is necessary in order to make the function in (\ref{eq:varphi_eps_def}) well-defined.

    We conclude the proof by defining $ S_\eps \coloneqq D\eps + S_\eps $, which satisfies $ v_\eps ( S_\eps ) = \bar{v}_\eps ( \bar{S}_\eps ) = \theta_\eps $ and 
    \begin{equation*}
        \frac{S_\eps }{\eps }
        =
        \frac{\bar{S}_\eps}{\eps }
        +
        D
        \leq M 
    \end{equation*}
    for $ M > 0 $ sufficiently large, only depending  on $ \alpha_-$, $ \kappa_0 $ , $ W $ and the convergence rate of $ \beta_\eps$ to $ b $.
\end{proof}
\begin{proposition}[$\liminf$-inequality]
\label{p:liminf 1d}
    Assume that the potential $ W $ is as in \Cref{sct:assumptions} and that the boundary data and weight satisfy (\ref{eq:bdry_assumptions_1d}) and $ \dw ( \alpha_\eps , \alpha), \dw ( \beta_\eps, b ) \in o (\eps) $. 
    If $ v_\eps $ is a sequence of minimizers of $ G_\eps^{(0)} $, then
    \begin{equation*}
    \liminf_{\eps \to 0 }
    G_\eps^{(2)} ( v_\eps ) 
    \geq 
    G^{(2)} ( b )
    =
    \omega'(0)
    \int_0^\infty 
    2 \sqrt{W ( z_\alpha ( s ) ) } z_\alpha ' ( s ) s \dd{s }
    .
    \end{equation*}
    Moreover we obtain that there exists a constant $ M' > 0 $ only dependent on $ W, \alpha_{-}, \norm{ \omega}_{ C^1 [0,T] } $, $ \kappa_0 $ and the speed of convergence $ \beta_\eps \to b $  such that for $ \eps > 0 $ sufficiently small, we have
    \begin{equation}
    \label{eq:liminf_1d_quantified}
        G_\eps^{(2)} ( v_\eps ) 
        \geq
        -M' \left ( 1 + \frac{ \dw ( \alpha_\eps , \alpha ) + \dw ( \beta_\eps , b ) }{\eps } \right).
    \end{equation}
\end{proposition}
\begin{proof}
    Let $ v_\eps $ be a sequence of minimizers and let $ S_\eps$ and $ \theta_\eps $ be as in \Cref{lem:properties_of_1d_minimizers} respectively \Cref{lem:properties_1d_mins_close_to_a} depending on the case we have in (\ref{eq:bdry_assumptions_1d}). 
    By using $ \omega ( t ) = \omega ( 0 ) + t \omega' ( 0 ) + R ( t ) $, we estimate
    \begin{align}
    \label{eq:Geps_first}
        G_\eps^{(2)} ( v_\eps ) 
        \geq &
        \left( 
        \int_0^{S_\eps}
         \frac{1}{\eps} W ( v_\eps ) + \eps (v_\eps ' )^2  \dd{ t }
        - 
        \mathrm{d}_{W} ( \alpha, b ) 
        \right)
        \frac{ \omega ( 0 )}{ \eps }
        \\
    \label{eq:Geps_second}
        &{}+ 
        \int_0^{S_\eps}
        \left( \frac{1}{\eps} W ( v_\eps ) + \eps (v_\eps ')^2 \right) t \dd{ t }
        \frac{ \omega' ( 0 ) }{ \eps }
        \\ &{} +
    \label{eq:Geps_third}
        \int_0^{S_\eps}
        \left( \frac{1}{\eps} W ( v_\eps ) + \eps (v_\eps ')^2 \right) \frac{R( t )}{\eps } \dd{ t }.
    \end{align}
    By the same arguments as in the $\limsup$-inequality and using that $ S_\eps / \eps$ stays uniformly bounded by \Cref{lem:properties_of_1d_minimizers} respectively \Cref{lem:properties_1d_mins_close_to_a}, the third summand (\ref{eq:Geps_third}) vanishes in the limit $ \eps \to 0 $.
    For the first summand (\ref{eq:Geps_first}) we apply Young's inequality and the triangle inequality for $ \dw $ to get
    \begin{align*}
        \left( 
        \int_0^{S_\eps}
         \frac{1}{\eps} W ( v_\eps ) + \eps (v_\eps ' )^2  \dd{ t }
        - 
        \mathrm{d}_{W} ( \alpha, b ) 
        \right)
        \frac{ \omega ( 0 )}{ \eps }
        &\geq 
        \left( 
        \int_0^{S_\eps}
         2 \sqrt{ W ( v_\eps) } v_\eps ' ( t )   \dd{ t }
        - 
        \mathrm{d}_{W} ( \alpha, b ) 
        \right)
        \frac{ \omega ( 0 )}{ \eps }
        \\
        & =
        \left( 
        \int_{\alpha_\eps}^{\theta_\eps}
         2 \sqrt{ W ( \rho ) } 
         \dd{ \rho }
        - 
        \int_{ \alpha }^{ b }
        2 \sqrt{ W ( \rho ) } \dd{ \rho}
        \right)
        \frac{ \omega ( 0 )}{ \eps }
        \\
        & \geq
        \left(
        -\dw ( \alpha_\eps, \alpha ) + \dw ( \alpha, \theta_\eps ) - \dw ( \alpha , b ) 
        \right)
        \frac{2 \omega ( 0 ) }{ \eps }
        \\
        & \geq
        \left( 
        -\dw ( \alpha_\eps , \alpha ) - \dw ( \theta_\eps ,b  ) 
        \right)
        \frac{2 \omega ( 0 ) }{ \eps },
    \end{align*}
    which vanishes by our choice of $ \theta_\eps $ and the assumption on $ \alpha_\eps $.
    Lastly we take care of the second summand (\ref{eq:Geps_second}). 
    By rescaling and using the notation from \Cref{lem:conv_of_rescaled}, we obtain that it is equal to
    \begin{equation*}
        \int_0^{S_\eps}
        \left( \frac{1}{\eps} W ( v_\eps ) + \eps (v_\eps ')^2 \right)t \dd{ t }
        \frac{ \omega' ( 0 ) }{ \eps }
        = 
        \int_0^{ S_\eps / \eps }
         \left( W ( w_\eps ) + (w_\eps ' )^2 \right) s
        \dd{ s }
        \omega' ( 0 ).
    \end{equation*}
    By \Cref{lem:properties_of_1d_minimizers} respectively \Cref{lem:properties_1d_mins_close_to_a}, we have that $ S_\eps / \eps \ $ stays bounded, and by \Cref{lem:conv_of_rescaled}, we know that $ w_\eps \to z_\alpha $ in $ \w^{1 , \infty }$ on every compact interval $ [0, N ] $. Additionally using that $ z_\alpha(t) = b $ for $ t \geq T^{(\alpha)}$, the second summand (\ref{eq:Geps_second}) converges in the limit $ \eps \to 0 $ to
    \begin{align*}
        \int_0^{S_\eps}
        t \left( \frac{1}{\eps} W ( v_\eps ) + \eps (v_\eps ')^2 \right) \dd{ t }
        \frac{ \omega' ( 0 ) }{ \eps }
        &\to 
        \int_{0 }^{ T^{(\alpha)} }
        ( W ( z_\alpha ) + (z_\alpha ')^2 ) s \dd{ s }
        \omega' ( 0 )
        \\ &
        =
        \int_{ 0 }^{ T^{(\alpha)} }
        2 \sqrt{ W ( z_\alpha(s) ) } z_\alpha ' ( s ) s \dd{ s } \omega' ( 0 ),
    \end{align*}
    which is what we wanted to show. 

    In order to obtain the lower bound (\ref{eq:liminf_1d_quantified}), we track the previous estimates and obtain
    \begin{align}
    \notag
        G_\eps^{(2)} ( v_\eps ) 
        \geq{}&   
        \frac{-2\omega ( 0 ) }{ \eps } 
        \left(
        \dw ( \alpha_\eps, \alpha ) + \dw ( \theta_\eps , b ) 
        \right)
        \\
        \label{eq:quant_second}
        &+
        \omega'( 0 ) \int_0^{S_\eps}  \left( W ( w_\eps ) + (w_\eps ')^2 \right) s
        \dd{ s }
        \\
        \label{eq:quant_third}
        & +
        \int_0^{S_\eps}
        \left(
        \frac{1}{\eps} W ( v_\eps ) 
        +
        \eps ( v_\eps ' )^2 
        \right)
        \frac{R(t)}{\eps}
        \dd{ s }.
    \end{align}
    We note that by inequality (\ref{eq:thetaeps_bound}) respectively (\ref{eq:thetaeps_bound_close_a}), we have due to the monotonicity of $ \dw $ and the asymptotic behaviour of $ \dw $ proven in \Cref{p:prop W} that
    \begin{align*}
        \dw ( \theta_\eps, b ) 
        \leq
         \dw ( b , b - C ( b - \beta_\eps ) ) + \dw ( b, b - C \eps^{1/(1+q)} )
        \leq
        C \dw ( \beta_\eps , b ) 
        +
        C \eps^{ (3+q)/(2(1+q))}.
    \end{align*}
    The latter summand is in $ o ( \eps ) $ due to $ q \in (0,1) $ and can in in the final estimate (\ref{eq:liminf_1d_quantified}) thus be absorbed into the constant. For the summand (\ref{eq:quant_second}), we note that by \Cref{p:upper bound derivative}  and inequality (\ref{eq:Seps_upper_bound}) respectively (\ref{eq:Seps_upper_bound_close_a}) it can be bounded from below by
    \begin{equation*}
    \omega'( 0 ) \int_0^{S_\eps} \left( W ( w_\eps ) + (w_\eps ')^2 \right) s
        \dd{ s }
        \geq
        -C \left(\frac{S_\eps}{\eps }\right)^2
        \geq
        -C M^2,
    \end{equation*}
    where $ M $ additionally may depend on the convergence rate of $ \beta_\eps $ to $ b $. The last summand (\ref{eq:quant_third}) is handled similarly by using the energy estimate (\ref{eq:energy_bound}). Combining these arguments yields the desired lower bound (\ref{eq:liminf_1d_quantified}), which finishes the proof.
\end{proof}

\section{Multidimensional case}
\label{sct:multidim_case}

Let $\Omega$, $ g $, $ g_\eps $ and the potential $ W  $ be as in \Cref{sct:assumptions}. We define the Cahn-Hilliard functional with Dirichlet data $ \mathcal{F}_\eps^{(0)} $ as in equation (\ref{e:Cahn-Hillard functionals definition}).
First we are going to discuss the properties of the map $ \Phi ( y, t ) = y + t \nu ( y ) $ which will be used with the area formula to reduce to the one-dimensional case. An illustration is given in \Cref{fig:reduction}.
\begin{figure}
    \centering

\tikzset{every picture/.style={line width=0.75pt}} 

\begin{tikzpicture}[x=0.75pt,y=0.75pt,yscale=-1,xscale=1]

\draw  [color={rgb, 255:red, 208; green, 2; blue, 27 }  ,draw opacity=1 ][fill={rgb, 255:red, 126; green, 211; blue, 33 }  ,fill opacity=1 ] (141.5,48.25) .. controls (203.5,18.75) and (337.5,17.25) .. (360.5,88.25) .. controls (383.5,159.25) and (314.88,117.38) .. (316,157.25) .. controls (317.13,197.13) and (517.5,186.25) .. (439.5,240.75) .. controls (361.5,295.25) and (208.5,278.25) .. (129,219.75) .. controls (49.5,161.25) and (79.5,77.75) .. (141.5,48.25) -- cycle ;
\draw  [dash pattern={on 0.84pt off 2.51pt}]  (82.6,123.4) -- (103,125.8) ;
\draw  [fill={rgb, 255:red, 248; green, 231; blue, 28 }  ,fill opacity=1 ] (300,55.25) .. controls (328,61.75) and (363.5,99.75) .. (298.5,141.25) .. controls (233.5,182.75) and (345,200.25) .. (402.5,217.25) .. controls (460,234.25) and (334,258.25) .. (283,253.25) .. controls (232,248.25) and (196,236.75) .. (154.5,211.75) .. controls (113,186.75) and (106,179.25) .. (103.5,158.25) .. controls (101,137.25) and (101,123.75) .. (113.5,98.75) .. controls (126,73.75) and (145.5,68.25) .. (180,56.25) .. controls (214.5,44.25) and (272,48.75) .. (300,55.25) -- cycle ;
\draw    (163,239.25) -- (184.93,204.44) ;
\draw [shift={(186,202.75)}, rotate = 122.22] [color={rgb, 255:red, 0; green, 0; blue, 0 }  ][line width=0.75]    (10.93,-3.29) .. controls (6.95,-1.4) and (3.31,-0.3) .. (0,0) .. controls (3.31,0.3) and (6.95,1.4) .. (10.93,3.29)   ;

\draw (81.5,67.4) node [anchor=north west][inner sep=0.75pt]    {$\Omega $};
\draw (171,188.4) node [anchor=north west][inner sep=0.75pt]    {$\nu $};
\draw (289,158.4) node [anchor=north west][inner sep=0.75pt]    {$\Omega _{\delta }$};
\draw (85.6,123.4) node [anchor=north west][inner sep=0.75pt]    {$\delta $};
\draw (182.4,131.6) node [anchor=north west][inner sep=0.75pt]    {$\Omega \setminus \Omega _{\delta }$};
\draw (358.5,71) node [anchor=north west][inner sep=0.75pt]    {$\partial \Omega $};

\end{tikzpicture}

    \caption{$ \Omega_\delta $ is the green subset of $ \Omega $ of points of distance smaller than $ \delta $ to its boundary. $ \Omega$ is the union of the green and yellow set and $\nu $ is the inwards unit normal to $ \partial \Omega $ marked as red.}
    \label{fig:reduction}
\end{figure}
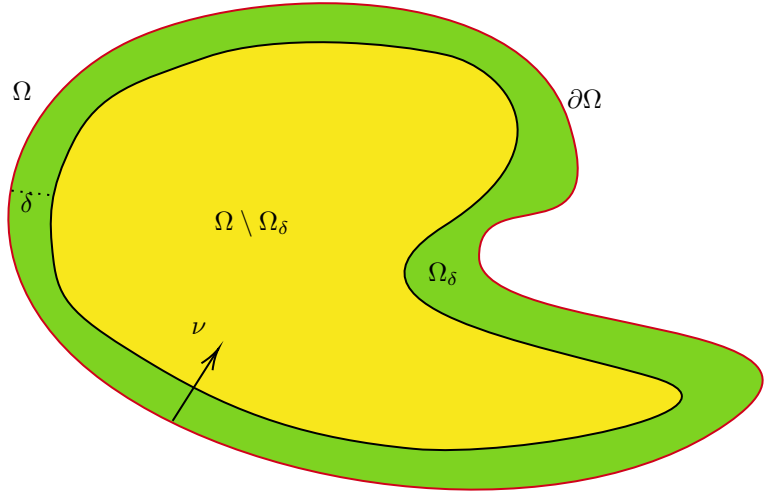

\begin{proposition}\label{p:The Phi}
For every $\delta>0$ small enough, depending on $\Omega$, define the mapping $\Phi:\partial \Omega \times (0,2\delta)\to \R^n$ by
\begin{equation*}
    \Phi(y,t)\coloneqq 
    y+t\nu(y)
\end{equation*}
where $\nu$ is the unit inward normal vector to $\partial \Omega$. Then $ \Phi $ is of class $C^{1}$, is injective with 
\begin{equation*}
   \det J_{\Phi_{\delta}}(y,t)>0
\end{equation*}
for every $(y,t)\in \partial \Omega\times (0,2\delta)$,
\begin{align*}
\Phi(\partial \Omega \times (0,\delta))&=\Omega_{\delta} \coloneqq \{x\in \Omega  :  0<\mathrm{dist}(x,\partial \Omega)<\delta\}
\shortintertext{and}
  \Phi(\partial \Omega \times \{\delta\}) &= \{x\in \Omega  : \mathrm{dist}(x,\partial \Omega)=\delta\}.
\end{align*}
Moreover, we have that $\Omega\setminus \overline{(\Omega_{\delta})}$ is connected, 
\begin{equation*}
    \det J_{\Phi}(y,0)=1 \quad \textup{for every $y\in \partial\Omega$}
\end{equation*}
and
\begin{equation}\label{e:mean curvature}
    \frac{\partial \det J_{\Phi}}{\partial t}(y,0)=-\kappa(y) \quad  \textup{for every $y\in \partial \Omega$},
\end{equation}
where $ J_\Phi $ is the Jacobian associated with $ \Phi $ and $\kappa:\partial \Omega\to \R$ is the scalar mean curvature of $\partial \Omega$.
\end{proposition}
The proof can be found in \cite[Lem.~2.6]{fonseca2025secondordergammalimitcahnhilliardfunctional1}.
Note that the scalar mean curvature is oriented and normalized such that the unit sphere in $ \R^n $ has constant positive mean curvature $ n - 1 $.

\begin{proposition}
\label{prop:zero_order_gamma_limit}
Let $\mathcal{F}^{(0)} \colon \mathrm{L}^1(\Omega)\to [0,\infty]$ be the functional given by
\begin{equation*}
    \mathcal{F}^{(0)}(u)=\int_\Omega W(u)\dd{x}.
\end{equation*}
Then we have
\begin{equation*}
    \Gamma \textup{-}\lim_{\eps \to 0} \mathcal{F}_{\eps}^{(0)}=\mathcal{F}^{(0)}.
\end{equation*}
\end{proposition}
Since the Proposition is no novel result and its proof straight forward, we have moved it to the appendix. 
The strategy for finding the recovery sequence is to choose a sequence whose gradient deteriorates, but slow enough so that  its contribution to the energy vanishes.

Having established the $0$-th order $ \Gamma $-limit, we see that $ \min \mathcal{F}^{(0)} = 0 $ by choosing any $ u \in \mathrm{L}^1 ( \Omega ; \{a,b\} ) $ as a minimizer. 
We consequently define $\mathcal{F}^{(1)}_{\eps}:\lp^1(\Omega)\to [0,\infty]$ by
\begin{equation*}
\mathcal{F}^{(1)}_{\eps}(u)=
\begin{cases}
    \int_\Omega \frac{W(u)}{\eps}+\eps|\nabla u|^2 \dd{x}, & \textup{if $u\in H^1(\Omega)$ and $u=g_{\eps}$ on $\partial \Omega$}, \\
    \infty, & \textup{otherwise}.
\end{cases}
\end{equation*}

\begin{proposition}\label{p:first order gamma limit multidim}
Let $\mathcal{F}^{(1)}:\mathrm{L}^1(\Omega)\to [0,\infty]$ be the functional given by 
\begin{equation*}
\mathcal{F}^{(1)}(u)=
\begin{cases}
    \frac{C_W}{b-a}|\diff u|(\Omega)+\int_{\partial \Omega}\dw(u,g)\dd{ \mathcal{H}^{n-1}}, & \textup{if $u\in \bv(\Omega;\{a,b\})$}, \\
    \infty , & \textup{otherwise}.
\end{cases}
\end{equation*}
Then 
\begin{equation*}
    \Gamma \textup{-}\lim_{\eps \to 0} \mathcal{F}^{(1)}_{\eps}=\mathcal{F}^{(1)}.
\end{equation*}
\end{proposition}
Again, this is not a novel result since it has been proven in \cite{RubinsteinSternberg90} and \cite{gazoulis_24_gamma_convergence}, albeit under slightly different assumptions. For the recovery sequence we use the idea that in the interior of $ \Omega $, the original result from Modica and Mortola \cite{ModicaMortola77} is used, while at the boundary, we use optimal transition profiles in normal directions. Our proof has been moved to the appendix since we want to focus on the second-order $ \Gamma $-limit.

We now want to discuss some consequences of the constant function $ b $ being the unique minimizer of $ \mathcal{F}^{(1)}$. For this we mainly use \cite[Thm.~4.1]{fonseca2025secondordergammalimitcahnhilliardfunctional}, which also holds up in the case of a subquadratic potential and states the following. 
If $ b $ is the unique minimizer of $ \mathcal{F}^{(1)} $, then at every point $ x \in \partial \Omega $ at which $ g ( x ) = a $, we must already have $ \kappa ( x ) \leq 0 $. In fact positive curvature at a point $ x \in \partial \Omega $ implies that we find a small neighbourhood $ A $ of $ x $ which has positive curvature. Then there exists an open set $ U \subseteq \Omega $ such that $ \partial U = A \cup B $ for some $ B \subseteq \Omega $ with $ \hm^{n-1} ( B ) < \hm^{n-1} ( A ) $. Consequently doing a transition along $ B $ and setting $ u = a $ on $ U $ yields a competitor with less energy. We have illustrated this in \Cref{fig:convexity_figure}.

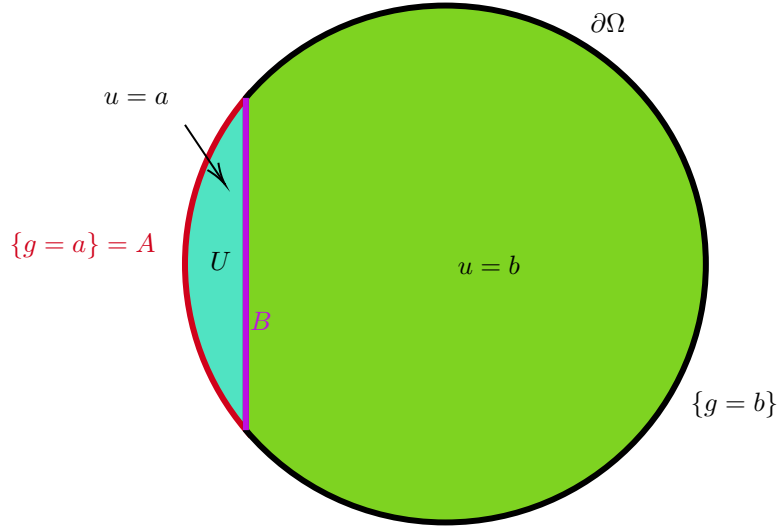
\begin{figure}
    \centering

\tikzset{every picture/.style={line width=0.75pt}} 

\begin{tikzpicture}[x=0.75pt,y=0.75pt,yscale=-1,xscale=1]

\draw  [color={rgb, 255:red, 208; green, 2; blue, 27 }  ,draw opacity=1 ][fill={rgb, 255:red, 80; green, 227; blue, 194 }  ,fill opacity=1 ][line width=2.25]  (220.41,243.56) .. controls (201.43,220.97) and (190,191.82) .. (190,160) .. controls (190,128.18) and (201.43,99.03) .. (220.41,76.44) -- cycle ;
\draw  [fill={rgb, 255:red, 126; green, 211; blue, 33 }  ,fill opacity=1 ][line width=2.25]  (220.41,76.44) .. controls (244.26,48.05) and (280.02,30) .. (320,30) .. controls (391.8,30) and (450,88.2) .. (450,160) .. controls (450,231.8) and (391.8,290) .. (320,290) .. controls (280.02,290) and (244.26,271.95) .. (220.41,243.56) -- cycle ;
\draw [color={rgb, 255:red, 189; green, 16; blue, 224 }  ,draw opacity=1 ][line width=2.25]    (220.41,76.44) -- (220.41,243.56) ;
\draw    (190,90) -- (208.89,118.34) ;
\draw [shift={(210,120)}, rotate = 236.31] [color={rgb, 255:red, 0; green, 0; blue, 0 }  ][line width=0.75]    (10.93,-3.29) .. controls (6.95,-1.4) and (3.31,-0.3) .. (0,0) .. controls (3.31,0.3) and (6.95,1.4) .. (10.93,3.29)   ;

\draw (101,142.4) node [anchor=north west][inner sep=0.75pt]  [font=\normalsize,color={rgb, 255:red, 208; green, 2; blue, 27 }  ,opacity=1 ]  {$\{g=a\} =A$};
\draw (201,152.4) node [anchor=north west][inner sep=0.75pt]    {$U$};
\draw (221,182.4) node [anchor=north west][inner sep=0.75pt]  [color={rgb, 255:red, 189; green, 16; blue, 224 }  ,opacity=1 ]  {$B$};
\draw (391,32.4) node [anchor=north west][inner sep=0.75pt]    {$\partial \Omega $};
\draw (441,222.4) node [anchor=north west][inner sep=0.75pt]    {$\{g=b\}$};
\draw (318,152.4) node [anchor=north west][inner sep=0.75pt]    {$ \begin{array}{l}
u=b\\
\end{array}$};
\draw (148,72.4) node [anchor=north west][inner sep=0.75pt]    {$u=a$};

\end{tikzpicture}

    \caption{The red curve $ A $ is longer than the magenta line $ B $, thus a transition along $ B $ would have less energy.}
    \label{fig:convexity_figure}
\end{figure}

\begin{proposition}
\label{prop:curvature_b-unqie_min}
    Assume that $ \partial \Omega $ is of class $ C^2 $ and that $ g \in C^2 ( \partial \Omega ;[a,b]) $ is such that the constant function $ b $ is the unique minimizer of the functional $ \mathcal{F}^{(1)}$. Then 
    \begin{equation*}
        \{ g = a \} \subseteq \{ \kappa \leq 0 \}.
    \end{equation*}
\end{proposition}
\begin{proof}
    By recalling the proof of \cite[Thm.~4.1]{fonseca2025secondordergammalimitcahnhilliardfunctional}, we see that the first difference of the subquadratic case compared to the quadratic case occurs when computing 
    \begin{align*}
        \dw ( b, \bar{g} ( x' ) ) - \dw ( a , \bar{g} ( x' ) )
        & =
        C_{\mathrm{W}} - 4 \int_a^{\bar{g} ( x' ) } 
        \sqrt{W (\rho)}  \dd{ \rho }
        \\
        &\leq
        C_{\mathrm{W}} - C \int_a^{\bar{g} ( x' ) } (\rho - a )^{ (1+q)/2} \dd{ \rho }
        \\
        & \leq
        C_{\mathrm{W}} - C ( \bar{g} ( x' ) - a ) ^{(3+q)/2}.
    \end{align*}
    If $ \bar{g} ( 0 ) = a $, then $ \nabla \bar{g} ( 0 ) = 0 $. By using $ g \in C^2 ( \partial \Omega )$, we get 
    \begin{equation*}
        \abs{ \bar{g} ( x' ) - a } \leq \norm{\bar{g}}_{C^2} \abs{x'}^2,
    \end{equation*}
    thus
    \begin{equation*}
        \dw ( b , \bar{g} ( x' ) ) - \dw ( a , \bar{g} ( x ' ))
        \leq 
        C_{\mathrm{W}} + \mathcal{O} ( \abs{x'}^{3+q} ).
    \end{equation*}
    Since $ q \in (0,1)$, the rest of the proof now works exactly the same.
\end{proof}
Using a simple compactness argument, we deduce the uniform version of the above proposition if we already assume the slightly stronger assumption $ \{ g = a \} \subseteq \{ \kappa < 0 \}$.
\begin{proposition}
\label{cor:boundary_data_curvature_relation}
    Assume that $ \partial \Omega $ is of class $ C^2 $, $ g $ is continuous and that $ \{ g = a \} \subseteq \{ \kappa < 0 \}$. Then there exists $ \kappa_0 < 0 $ and $ \alpha_- \in (a, b] $ such that
    \begin{equation*}
        \{ g \leq \alpha_- \}
        \subseteq 
        \{ \kappa < \kappa_0 \}.
    \end{equation*}
\end{proposition}
\begin{proof}
    For $ n \in \N $ define 
    \begin{equation*}
        K_n \coloneqq \{ y \in \partial \Omega \colon g ( y ) \leq a + \frac{1}{n} \}.
    \end{equation*}
    $ K_n $ is a decreasing sequence of compact sets such that $ \bigcap K_n = \{ g = a \} \subseteq \{ \kappa < 0 \}$. Since the right-hand side in this inclusion is an open subset of $ \partial \Omega$, we must find $ m\in \N $ such that $ K_m \subseteq \{ \kappa < 0 \} $. Now define the open sets
    \begin{equation*}
        U_n \coloneqq \{ y \in \partial \Omega \colon \kappa < -\frac{1}{n} \}.
    \end{equation*}
    Then $ K_m \subseteq \bigcup U_n $, and by compactness of $ K_m $ we deduce that there must exist $ l \in \N $ such that $ K_m \subseteq U_l$. Setting $ \alpha_- = a+1/m $ and $ \kappa_0 = -1/l $ finishes the proof.
\end{proof}

\subsection{Upper bound inequality for the second order functionals}

As mentioned before, our key assumption is that the constant function $ b $ is the unique minimizer of the first order $ \Gamma $-limit $ \mathcal{F}^{(1)}$, which has been computed in \Cref{p:first order gamma limit multidim}.
Consequently we define for every $\eps\in (0,\infty)$ the functional $\mathcal{F}^{(2)}_{\eps}:\mathrm{L}^1(\Omega)\to [0,\infty]$ by 
\begin{equation*}
\mathcal{F}^{(2)}_{\eps}(u)=
\begin{cases}
    \int_{\Omega} \frac{W(u)}{\eps^2}+|\nabla u|^2 \dd{x}-\frac{1}{\eps}\int_{\partial \Omega}\textup{d}_W(b,g)\dd{ \mathcal{H}^{n-1} }, & \textup{if $u\in \h^1(\Omega)$ and $u=g_{\eps}$ on $\partial \Omega$}, \\
    \infty, & \textup{otherwise}.
\end{cases}
\end{equation*}
We will start by showing the construction for the recovery sequence. The idea is to use the same construction as in the one-dimensional case along directions of the inner unit normal.

\begin{theorem}
\label{prop:recovery_seq_second_order}
Assume that $W$ is such that $u_0:\equiv b$ is the unique minimizer of $\mathcal{F}^{(1)}$. Then there exists a sequence $(u_{\eps})_\eps$ in $\mathrm{L}^1(\Omega)$ such that 
\begin{equation*}
    \limsup_{\eps\to 0}\mathcal{F}^{(2)}(u_{\eps})\leq
    \int_{ \partial \Omega }
    - \kappa ( y ) 
    \int_0^\infty \dw ( z_{g(y)} ( s ) , b ) \dd{ s }
    \dd{ \hm^{n-1} ( y )}
    \eqqcolon 
    G^{(2)} ( b )
\end{equation*}
and $u_{\eps}\to u_0$ in $\mathrm{L}^1(\Omega)$ as $\eps\to 0$.
\end{theorem}

\begin{proof}
    The construction is similar as in \cite{fonseca2025secondordergammalimitcahnhilliardfunctional1}, but we can simplify some steps due to the local integrability of $ W^{-1/2}$.
    Define $ \Psi_\alpha $ as in equation (\ref{eq:def_Psialpha}) and let $ T_\eps ( y ) \coloneqq \eps \Psi_{g_\eps ( y ) } ( b ) $. Note that there exists a constant $ C > 0 $ such that $ T_\eps ( y ) \leq C \eps $ for all $ y \in \partial \Omega $ due to $ g_\eps (y ) \in [a,b]$. Then let
    \begin{equation*}
        v_\eps (y, \cdot ) \colon
        [ 0, T_\eps ( y ) ] \to [ g_\eps ( y ) , b ]
    \end{equation*}
    be defined as $ z_{g_\eps ( y ) } ( \cdot / \eps )$, which satisfies  $ v_\eps ( y , 0 ) = g_\eps ( y ) $, $ v_\eps ( y , T_\eps ( y ) ) = b $ and
    \begin{equation*}
        \partial_t v_\eps ( y , t ) 
        =
        \frac{ \sqrt{ W ( v_\eps ( y, t ) ) } }{ \eps }.
    \end{equation*}
    Moreover we denote, as in Lemma \ref{lem:conv_of_rescaled}, by $ w_\eps ( y, s ) \coloneqq v_\eps ( y, \eps s )$ the rescaling of $v_\eps $. We first assume that $ g_\eps $ is continuously differentiable. Then by the inverse function theorem, so is $ v_\eps $.
    We extend $ v_\eps ( y, t ) $ to be equal to $ b $ for $ t \geq T_\eps ( y ) $. To understand the tangential derivatives of $ v_\eps$, we note that for every $ y \in \partial \Omega, $ we have
    $ v_\eps ( y, \Psi_\eps ( y , r ) ) = r $. Let $ z $ be a tangent vector to $ y $ with respect to $ \partial \Omega $. By differentiating the equality, we obtain
    \begin{equation}
    \label{eq:est_for_tang_der_veps}
        \partial_z v_{\eps}( y, \eps\Psi_{ g_\eps ( y ) } ( r ) ) =- \partial_t v_\eps ( y, \eps\Psi_{g_\eps} (y, r ) ) \eps \partial_z (\Psi_{g_\eps ( y )} ( r ) ) 
        =
        \frac{ \sqrt{ W ( v_\eps ( \Psi_{g_\eps (y)} ( r ) ) ) }}{ \sqrt{ W ( g_\eps ( y ) ) }} \partial_z g_\eps ( y ).
    \end{equation}
    Note that if $ g_\eps(y) \in \{a,b\}$, then $g_\eps $ already assumes an extremum in $ x $ and thus $ \partial_z g_\eps =0 $ so that the above expression is set to zero in those cases.
    By a density argument (both approximating $ g $ via continuously differentiable $ g_n $ and $ 1 /\sqrt{W}$ by $ 1/\sqrt{W + 1/n}$), equality (\ref{eq:est_for_tang_der_veps}) still holds $\mathcal{H}^{n-1}$-a.e. if we assume only $ g_\eps \in \h^1 ( \partial \Omega )$ since via assumption, (\ref{eq:est_for_tang_der_veps}) is square integrable, see also estimate (\ref{eq:estim_for_B}) below. We finally set 
    \begin{equation*}
        u_\eps ( x )
        \coloneqq
        \begin{cases}
            v_\eps ( \Phi^{-1 } ( x ) ), & \text{if }x \in \Omega_\delta,
            \\
            b, & \text{else},
        \end{cases}
    \end{equation*}
    with $\delta$ and $\Phi$ as in Proposition \ref{p:The Phi}. Let $ y \colon \Omega_\delta \to \partial \Omega $ be the projection onto the boundary. 
    From $ \Phi^{-1 } ( x ) = (y( x ) , \mathrm{dist} ( x, \partial \Omega ) ) $ and $ \abs{ \nabla \mathrm{dist} } \leq 1  $, we deduce that
    \begin{equation*}
        \abs{ \nabla u_\eps ( x ) }^2
        \leq
        \abs{ \partial_t v_\eps ( \Phi^{-1} ( x ) ) }^2
        +
        C \norm{ \nabla y }_{ L^\infty}^2
        \abs{ \nabla_\tau v_\eps ( \Phi^{-1} ( x ) ) }^2,
    \end{equation*}
    where $ \nabla_\tau $ denotes the tangential derivative. We define $ \omega ( y, t ) \coloneqq \det J_\Phi ( y, t )$ and, by the Area Formula and Fubini's Theorem, we obtain the estimate
    \begin{align*}
        \mathcal{F}_\eps^{(2 ) } ( u_\eps )
         ={} &
        \int_{ \partial \Omega }
        \int_0^{\delta } 
        \left( \frac{1 }{\eps^2} W ( v_\eps ( y, t ) ) + \abs{ \nabla u_\eps ( \Phi ( y, t ) ) }^2 
        \right) \omega ( y, t ) 
        \dd{ t } 
        \dd{ \mathcal{H}^{N-1} (y )}
        -
        \int_{ \partial \Omega } \frac{\dw( g ( y ) , b ) }{\eps}
        \dd{ \mathcal{H}^{N- 1 } ( y ) }
        \\
        \leq{} &
        \int_{ \partial \Omega }
        \int_0^\delta 
        \left( \frac{ 1 }{\eps^2 } W ( v_\eps ( y, t ) ) + \abs{ \partial_t v_\eps ( y, t ) }^2 \right)
        \omega( y, t ) 
        \dd{ t }
        -
        \frac{ 1 }{\eps } 
        \dw ( g ( y ) , b )
        \dd{ \mathcal{H}^{N- 1 } ( y ) }
        \\
        &+
        C \norm{ \nabla y }_{ \lp^\infty ( \Omega_\delta ) }^2
        \int_{ \partial \Omega } \int_0^\delta 
        \abs{ \nabla_\tau v_\eps ( y, t ) }^2 
       \omega(y,t) \dd{ t } \dd{ \mathcal{ H }^{N-1} ( y ) }
        \eqqcolon \mathcal{ A } + \mathcal{ B }.
    \end{align*}
    For $ \mathcal{A}$, we apply the methods of the one-dimensional case and write
    \begin{align}
    \label{eq:A_first_summand}
        \mathcal{A}
        ={} &
        \int_{ \partial \Omega } 
        \int_0^{T_\eps ( y ) }
        \frac{1}{\eps^2} W ( v_\eps ( y , t ) ) + \abs{ \partial_t v_\eps ( y, t ) }^2 \dd{ t }
        -
        \frac{1}{\eps} \mathrm{d}_W ( g( y ), b )
        \dd{ \mathcal{H}^{N-1} ( y ) }
        \\
        \label{eq:A_second_summand}
        & +
        \int_{ \partial \Omega }
        -\kappa ( y ) \int_0^{T_\eps ( y ) / \eps }
        2\sqrt{ W ( w_\eps ( y, s ) ) }  \partial_s w_\eps ( y, s )
        s
        \dd{ s }
        \dd{ \hm^{n-1}}
        \\
        \label{eq:A_third_summand}
        & +
        \int_{ \partial \Omega }
        \int_0^{T_\eps ( y ) } 
        \left( 
        \frac{1}{\eps} W ( v_\eps ( y, t ) ) + \eps ( \partial_t v_\eps ( y, t ) )^2
        \right)
        \frac{ R( t ) }{\eps }
        \dd{ t }
        \dd{ \mathcal{H}^{N-1} ( y ) }.
    \end{align}
    The last summand (\ref{eq:A_third_summand}) vanishes since $ T_\eps ( y ) \leq C \eps $, $ \max_{[0,C\eps ]}\abs{R}/\eps \to 0 $ as $ \eps \to 0 $ and by the uniform energy bound proven in \Cref{prop:minimizer_properties}.
    For the second summand (\ref{eq:A_second_summand}), we note that $T_\eps ( y ) /\eps \to T_{g( y ) }$ almost everywhere. 
    Additionally using \Cref{p:upper bound derivative} to get $ \partial_s w_\eps \leq C $ and $ T_\eps (y ) /\eps \leq C $, we obtain by dominated convergence and integration by parts that the second summand converges to
    \begin{align*}
    &\lim_{\eps \to 0}
    \int_{ \partial \Omega }
        -\kappa ( y ) \int_0^{T_\eps ( y ) / \eps }
        2\sqrt{ W ( w_\eps ( y, s ) ) }  \partial_s w_\eps ( y, s )
        s
        \dd{ s }
        \dd{ \hm^{n-1} (y)}
        \\
        ={}&
        \int_{ \partial \Omega} 
        -\kappa ( y ) 
        \int_{0}^{T_g( y ) }
        2\sqrt{ W ( z_{g( y ) } ( s )) }
        z_{ g( y ) }' ( s ) s 
        \dd{ s }
        \dd{ \hm^{n-1}(y) }
        \\
        ={}&
        \int_{ \partial \Omega }
        -\kappa ( y ) \int_0^\infty \dw ( z_{g( y ) }(s),b)
        \dd{ s }
        \dd{ \hm^{n -1} ( y ) }.
    \end{align*}
    For the first summand (\ref{eq:A_first_summand}), we note that it is estimated from above by
    \begin{equation*}
        \int_{ \partial \Omega }
        \frac{ \abs{\mathrm{d}_W ( g_\eps ( y ) , b ) - \mathrm{d}_W ( g( y ) , b ) } }{ \eps }
        \dd{ \mathcal{H}^{N-1} ( y ) }
        =
        \int_{ \partial \Omega }
        \frac{ \mathrm{d}_W ( g_\eps( y ) , g( y ) ) }{ \eps }
        \dd{ \mathcal{ H}^{ N -1 } ( y ) },
    \end{equation*}
    which vanishes by assumption (\ref{eq:geps_to_g}). 

    For the error term $ \mathcal{B}$, we lastly note that by inequality (\ref{eq:est_for_tang_der_veps}), it can be estimated by
    \begin{align}
     \label{eq:estim_for_B}
        \mathcal{B}
        & \lesssim
        \int_{ \partial \Omega }
        \frac{\abs{ \nabla_\tau g_\eps ( y ) }^2}{W(g_\eps(y))} \int_0^{\delta } W (v_\eps (y,t) ) \dd{t} \dd{ \mathcal{H}^{N-1} (y)}
        \lesssim
        \eps
        \int_{ \partial \Omega }
        \frac{\abs{ \nabla_\tau g_\eps ( y ) }^2 }{W(g_\eps ( y ) ) }
        \dw( g_\eps (y ) , b ) 
        \dd{ \mathcal{H}^{N-1}( y ) }.
    \end{align}
    This vanishes as $ \eps \to 0 $ by assumption (\ref{eq:ass_tang_deriv_squared}), which finishes the proof.
\end{proof}
We now turn to the lower bound. In the one-dimensional case, we were able to precisely predict the short-time behaviour of minimizers of the weighted functional by studying the corresponding variational equation. While much harder in higher dimensions, a similar idea still applies. We first consider the corresponding Euler--Lagrange equation.

\begin{lemma}
Let $\eps>0$ and $u_{\eps}\in \argmin\mathcal{F}_{\eps}^{(0)}$. Then $u_{\eps}\in C^2(\Omega)$, $a\leq u_{\eps} \leq b$ and
\begin{equation}\label{e:Euler-Lagrange multidimensional}
    \eps^2 \Delta u(x)=W'(u(x))   
\end{equation}
for every $x\in \Omega$. Moreover, there exists a constant $C$, that depends only on $n$, $ \Omega $ and $W$, such that  
\begin{equation}\label{e:L^inf bound of the gradient}
    |\nabla u(x)| \leq \frac{C}{\eps}
\end{equation}
for every $x\in \Omega$.
\end{lemma}

\begin{proof}
By a simple truncation argument we have that $a\leq u(x)\leq b $ for $\calL^{n}$-a.e. $x\in \Omega$. From the De Giorgi's Theorem we have that there exists $\gamma \in (0,1)$ such that $u\in C^{0,\gamma}_{\textup{loc}}(\Omega)$ and, thanks to Proposition \ref{p:prop W}, $W'(u)\in C^{0,q\gamma}_{\textup{loc}}(\Omega)$. Moreover, considering outer variations of $u$, we can deduce that $u$ satisfies the following Euler--Lagrange equation
  \begin{equation*}
      \eps^2 \textup{div}(\nabla u)=W'(u)
  \end{equation*}
in a distributional sense. From the regularity theory of the Poisson equation we obtain that $u\in C^{2,q\gamma}_{\textup{loc}}(\Omega)$ and hence $u$ satisfies \eqref{e:Euler-Lagrange multidimensional} in a strong sense. Equation \eqref{e:L^inf bound of the gradient} is a direct consequence of $a\leq u\leq b$ and \cite[Lem.~A.1, Lem.~A.2]{bethuel_brezis_helein_93_asymptotics_for_ginzburg_landau}.
\end{proof}

The following Proposition is the multi-dimensional equivalent of \Cref{lem:properties_of_1d_minimizers} respectively \Cref{lem:properties_1d_mins_close_to_a} and tells us that minimizers of $ \mathcal{F}_\eps^{(0)} $ make a sharp transition from $ g_\eps $ to $ b $ close to the boundary. It relies on the work \cite[Prop.~4.1]{sternberg_zumbrun_98_connectivity_of_phase_boundaries} and \cite{caffarelli_cordoba_95_unfirom_convergence_of_a_singular_perturbation_problem} and was proven in \cite[Thm.~4.9]{fonseca2025secondordergammalimitcahnhilliardfunctional1}.

\begin{lemma}\label{p:main estimate liminf}
Assume that $u_0 \coloneqq b$ is the unique minimizer of the functional $\mathcal{F}^{(1)}$. Then there exist two constants $C,\mu >0$ such that for every $\delta>0$ small enough and every $\eps> 0$ sufficiently small, depending on $\delta$, we have
\begin{equation*}
    b-u_{\eps}(x)\leq Ce^{-\mu\delta/\eps} \quad \textup{for every $x\in \Omega \setminus \Omega_{2\delta}$}
\end{equation*}
for any $u_{\eps}\in \argmin\mathcal{F}^{(1)}_{\eps}$.
\end{lemma}

\begin{proof}
The proof is as in \cite[Thm.~4.9]{fonseca2025secondordergammalimitcahnhilliardfunctional1}. The only change is that from the Euler--Lagrange equation, we deduce that for a minimizer $ u_\eps \in \argmin \mathcal{F}_\eps^{(0)} $, we have for $ u_\eps $ close to $b$ that
\begin{equation*}
    \eps^2 \Delta ( b - u_\eps ) (x)
    =
    -W' (u_\eps ( x ) )
    \geq
    C ( b - u_\eps ( x ) )^q
    \geq
    C ( b - u_\eps ( x ) ).
\end{equation*}
The last inequality follows since $ q \in (0,1)$ and $ b-u_\eps \leq b-a $ by \Cref{prop:minimizer_properties}. Now we have arrived at the same inequality  which is used in the reference material, and the rest of the proof is identical.
\end{proof}
We finish by showing the limit inferior inequality for the second order functional by applying slicing, using the lower bounds of the one-dimensional functional to pull the limit into the integral and then using the one-dimensional limit inferior inequality.

\begin{theorem}
\label{thm:liminf_second_order}
Assume that $u_0 \coloneqq b$ is the unique minimizer of the functional $\mathcal{F}^{(1)}$. Under the assumptions from \Cref{sct:assumptions}, we have
\begin{equation*}
    \liminf_{\eps\to 0} \mathcal{F}^{(2)}_{\eps}(u_{\eps})\geq \int_{\partial \Omega}-\kappa(y)
    \int_0^{\infty}\dw ( z_{g( y ) } (s), b ) \dd{s} \dd{\mathcal{H}^{n-1}(y)},
\end{equation*}
where $u_{\eps}\in \argmin \mathcal{F}^{(0)}_{\eps}$ for every $\eps>0$.
\end{theorem}

\begin{proof}
Fix $\delta > 0 $ as in Proposition \ref{p:The Phi}. By Proposition \ref{p:main estimate liminf} we find $C,\mu>0$, independent of $\eps>0$, such that 
 \begin{equation}\label{e:main estimate eq}
    b-u_{\eps}(x)\leq Ce^{-\mu\delta/\eps} \quad \textup{for every $x\in \Omega \setminus \Omega_{\delta}$},
\end{equation}
for every $\eps>0$ sufficiently small. Write
\begin{align}\label{e:first eq last teo}
    \mathcal{F}^{(2)}_{\eps}(u_{\eps})\geq \int_{\Omega_{\delta}}\left(\frac{W(u_{\eps})}{\eps^2}+|\nabla u_{\eps}|^2\right) \dd{x} -\frac{1}{\eps}\int_{\partial \Omega}\dw(g,b)\dd{ \mathcal{H}^{n-1}}
    .
\end{align}
By the Area Formula, Fubini and \Cref{p:The Phi} we get 
\begin{align*}
 & \int_{\Omega_{\delta}}\left(\frac{W(u_{\eps})}{\eps^2}+|\nabla u_{\eps}|^2\right) \dd{x} \\
  ={}&
  \int_{\partial \Omega}\int_0^{\delta}\left(\frac{W(u_{\eps}(\Phi(y,t)))}{\eps^2}+|\nabla u_{\eps}(\Phi(y,t))|^2)\right)\omega(y,t) \dd{t} \dd{\hm^{n-1}(y)},
\end{align*}
where as before $\omega(y,t)\coloneqq \det J_{\Phi}(y,t)$. 
Let $\tilde u_{\eps} \coloneqq u_{\eps}\circ \Phi$. Since $u_{\eps}\in C^1(\Omega)$ and $\Phi(y,t)=y+t\nu(y)$, we have that 
\begin{equation*}
    \frac{\partial \tilde u_{\eps}}{\partial t}(y,t)=\frac{\partial u_{\eps}}{\partial \nu(y)}(y+t\nu(y))
\end{equation*}
and therefore 
\begin{equation*}
    |\nabla u_{\eps}(\Phi(y,t))|^2\geq  \left|\frac{\partial \tilde u_{\eps}}{\partial t}(y,t)\right|^2
\end{equation*}
for every $y\in \partial \Omega$ and every $t\in (0,\delta)$. In particular we obtain that 
\begin{equation}\label{e:second eq last teo}
    \mathcal{F}_\eps^{(2)} (u_\eps)
    \geq 
    \int_{\partial \Omega}
    \left(
        \int_0^{\delta}
        \left(
            \frac{W(\tilde u_{\eps}(y,t))}{\eps^2}+\left|\frac{\partial \tilde u_{\eps}}{\partial t}(y,t)\right|^2
        \right)\omega(y,t) 
        \dd{t}
        - \frac{1}{\eps}\dw(g(y),b)
    \right) 
    \dd{ \mathcal{H}^{n-1}(y) }
\end{equation}
and thanks to inequality \eqref{e:main estimate eq}, for every $y\in \partial \Omega$ we have that 
\begin{equation}\label{e:convergence to b}
    b-Ce^{-\mu\delta/\eps}\leq \tilde u_{\eps}(y,\delta)\leq b.
\end{equation}
Let $v^{y}_{\eps}\in \h^1((0,\delta))$ be a minimizer of the weighted 1-dimensional functional 
\begin{equation*}
    v\mapsto \int_0^{\delta}\left(\frac{W(v(t))}{\eps^2}+|v'(t)|^2\right)\omega(y,t)\dd{t},
\end{equation*}
defined for all $v\in \h^1((0,\delta))$ such that $v(0)=g_{\eps}(y)$ and $v(\delta)=\tilde u(y,\delta) \in [b-C e^{-\mu \delta / \eps }, b ]$. From \eqref{e:convergence to b} and \Cref{p:liminf 1d}, which can be applied since we assume (\ref{eq:assumption_bdry_data_distinction}), we thus deduce that for $ \hm^{n-1}$-almost every $ y \in \partial \Omega $, we have
\begin{align}\notag
& 
\liminf_{\eps\to 0} 
\int_0^{\delta}
\left(
    \frac{W(\tilde u_{\eps}(y,t))}{\eps^2}+\left|\frac{\partial \tilde u_{\eps}}{\partial t}(y,t)\right|^2
\right)\omega(y,t) 
\dd{t}
- 
\frac{1}{\eps}\textup{d}_W(g(y),b) 
\\ 
\notag
\geq{}& 
\liminf_{\eps\to 0} 
    \int_0^{\delta}\left(\frac{W(v^{y}_{\eps}(t))}{\eps^2}+|(v^{y}_{\eps})'(t)|^2\right)\omega(y,t)\dd{t}- \frac{1}{\eps}\dw(g(y),b)
\\ 
\label{e:third eq}
\geq{} & 
2 \frac{\partial \omega}{\partial t}(y,0) \int_0^{\infty}\sqrt{W(z_{g(y)}(s))}z_{g(y)}'(s) s \dd{s}
=
-\kappa(y)\int_0^{\infty}\dw ( z_{g ( y ) }(s), b)\dd{s}, 
\end{align}
where in the last equality we used equation \eqref{e:mean curvature}. To conclude the proof, we have to justify pulling the limit inferior into the integral. By \Cref{p:liminf 1d} we have that
\begin{align}
\notag
    & \int_0^{\delta}
    \left(
        \frac{1}{\eps^2}W(v^{y}_{\eps}(t))+|(v^{y}_{\eps})'(t)|^2
    \right)\omega(y,t)\dd{t}
    - \frac{1}{\eps}\dw(g(y),b)
    \\
    \label{eq:majorant}
    \geq{} &
    -M' \left( 
    1 + \frac{ \dw ( g_\eps ( y ) , g ( y ) ) + \dw ( b - C e^{-\delta \mu / \eps } , b ) }{ \eps }
    \right),
\end{align}
where $ M' $ possibly depends on $ W $, $ \Omega$, $ \alpha_-$, $ \kappa_0 $ and the speed of convergence of $ C e^{ \delta \mu / \eps }$. These quantities are all independent of $ y \in \partial \Omega $. Furthermore by assumption (\ref{eq:geps_to_g}), the function in (\ref{eq:majorant}) is integrable on $ \partial \Omega $ with respect to $ \hm^{n-1}$. Thus we can apply Fatou's Lemma to argue with \eqref{e:second eq last teo} and \eqref{e:third eq} that
\begin{equation*}
    \liminf_{\eps \to 0 }
    \mathcal{F}_\eps^{(2)} ( u_\eps )
    \geq
    \int_{\partial \Omega }
    - \kappa ( y ) 
    \int_0^\infty \dw ( z_{g(y)}(s), b ) \dd{ s },
\end{equation*}
which is what we wanted to show.
\end{proof}

\begin{remark}
    We lastly want to discuss the different scaling regimes and argue that the order $ \eps $ is the first non-trivial scaling for the second-order functional. Namely let $ \delta_\eps \to 0 $ be any other scale such that $ \eps / \delta_\eps \to 0 $. It follows then from our arguments that as long as 
    \begin{equation}
    \label{eq:scalign_assumption}
    \lim_{\eps \to 0}
        \int_{\partial \Omega }
        \frac{\dw ( g_\eps , g ) }{\delta_\eps }
        \dd{ \hm^{n-1} }
        =
        0,
    \end{equation}
    it holds that the $ \Gamma $-limit of $ (\mathcal{F}_\eps^{(1)}-\min \mathcal{F}^{(1)})/\delta_\eps $ is given by zero. When we drop assumption (\ref{eq:scalign_assumption}) however, we could arrive at a different $ \Gamma $-limit.
\end{remark}

\section{Appendix}
\label{sct:appendix}

\begin{proof}[Proof of \Cref{prop:zero_order_gamma_limit}]
    The $ \liminf $-inequality follows immediately from Fatou's Lemma by passing to a suitable subsequence which realizes the limit inferior of the energies and converges pointwise almost everywhere.

    For the $\limsup $-inequality, take any $ u \in \lp^1 ( \Omega ) $. We note that as an assumption on the boundary data we will only need that $ \partial \Omega $ is Lipschitz and that $ g_\eps \in \h^{1/2} ( \partial \Omega ) $ with 
    \begin{equation}
    \label{eq:zero_order_ass}
        \eps^2 [ g_\eps ]_{\mathrm{H}^{1/2} ( \partial \Omega ) }
        \coloneqq
        \eps^2 \int_{ \partial \Omega } \int_{ \partial \Omega }
        \frac{ \abs{ g_\eps ( x ) - g_\eps ( y )}^2 }{\abs{x-y}^n}
        \dd{ \hm^{n-1} ( x ) } 
        \dd{ \hm^{n-1 } ( y ) }
        \to 
        0,
    \end{equation}
    which is a strictly milder assumption than (\ref{eq:ass_tang_deriv_squared}) because $ [ \cdot ]_{ \mathrm{H}^{1/2} ( \partial \Omega ) } \lesssim \norm{ \nabla \cdot }_{ \mathrm{L}^2 ( \partial \Omega ) } $. Since $ g_\eps \in \mathrm{H}^{1/2}( \partial \Omega ) $, there exist $ v_\eps \in \mathrm{H}^1 ( \R^n )$ such that $ \tr ( v_\eps ) = g_\eps $ on $ \partial \Omega $, $v_\eps ( x ) \in [a,b] $ for almost every $ x \in \R^n $ and $ \norm{ v_\eps }_{ \mathrm{H}^1 ( \R^n ) } \lesssim \norm{ g_\eps }_{ \mathrm{H}^{1/2} ( \partial \Omega ) } $.
    Moreover by a standard approximation result and passing to an appropriate subsequence, we find $ w_\eps \in \mathrm{H}^1 ( \Omega ) $ such that $ w_\eps \to u $ pointwise almost everywhere and in $ \mathrm{L}^1 ( \Omega ) $ and with the property that
    \begin{equation*}
        \eps^2 \int_\Omega \abs{ \nabla w_\eps }^2 \dd{ x }
        \to 
        0.
    \end{equation*} 
    Let $ \eta_\eps $ be a sequence of cutoff-functions such that $ \abs{ \nabla \eta_\eps } \lesssim 1/\eps $, $ \mathrm{supp} ( \eta_\eps ) \subset\joinrel\subset \Omega $ and $ \eta_\eps \to 1 $. Finally define $ u_\eps \coloneqq (1-\eta_\eps ) v_\eps + \eta_\eps w_\eps  \in \mathrm{H}^1 ( \Omega ) $. Since $ \eta_\eps \to 1$, it follows that $ u_\eps \to u $ in $ \mathrm{L}^1 ( \Omega ) $ and pointwise almost everywhere. Furthermore due to $ \abs{ v_\eps }, \abs{ w_\eps } \leq \max\{ a, b \}$, we have
    \begin{align*}
        \eps^2 \int \abs{ \nabla u_\eps } \dd{ x }
        \less
        \eps^2 \int \abs{ \nabla v_\eps }^2 \dd{ x }
        +
        \eps^2 \int \abs{ \nabla w_\eps }^2 \dd{ x }
        +
        C \eps
        \to 0.
    \end{align*}
    Thus by the dominated convergence theorem, we conclude that
    \begin{equation*}
        \limsup_{ \eps \to 0 }
        \mathcal{F}_\eps^{ ( 0 ) } (u_\eps )
        =
        \mathcal{F}^{(0)} (u),
    \end{equation*}
    which finishes the proof.
    \end{proof}

\begin{lemma}
\label{lem:approx_near_boundary}
    Let $ \Omega $ be a bounded and open set with Lipschitz boundary. Then for every $ u \in \bv ( \Omega ; \{a,b \} ) $, we find a sequence $ (u_n)_{n \in \N } \subseteq \bv ( \Omega ; \{a,b\} ) $ such that $ u_n \to u $ in the strict convergence on $ \bv ( \Omega ) $ and such that for all $ n \in \N $ exists some $ \delta_n \gtr 0 $ sufficiently small such that $ u_n ( \Phi (t, y ) ) = \tr u_n ( y )  $ for $ \hm^{ N - 1 }$-almost every $ y \in \partial \Omega $ and every $ 0 \less t \less \delta_n $.
\end{lemma}

\begin{proof}
    By the slicing of $ \bv $-functions (see end of Section 3.11 in \cite{ambrosio_fusco_pallara_00_functions_of_bv}) and recalling \Cref{p:The Phi}, we find a sequence $ t_n \downarrow 0 $ such that 
    \begin{equation}
    \label{eq:slice_bound}
    \sup_{n \in \N } \abs{ \diff_\tau (u \circ \Phi ) ( \cdot, t_n ) } ( \partial \Omega ) \less \infty,
    \end{equation} 
    where $ \diff_\tau $ denotes the tangential derivative of $ u \circ \Phi $ along $ \partial \Omega_t$. We then define
    \begin{equation*}
        u_n ( x ) 
        \coloneqq
        \begin{cases}
            u( x ), &\text{if } x \in \Omega \setminus \Omega_{t_n},
            \\
            \tr u ( \Phi (y, t_n ) ) , &\text{if } x = \Phi (y, t ) \text{ for some } t \in (0,t_n ), y \in \partial \Omega.
        \end{cases}
    \end{equation*}
    Here the trace is taken from inside the set $ \Omega \setminus \Omega_{t_n}$. 
    We immediately get by the dominated convergence theorem that $ u_n \to u $ in $ \lp^1 ( \Omega ) $. We note that since the traces of $ u_n $ agree on both sides of $ \partial \Omega_{t_n}$, we have that
    \begin{equation*}
        \abs{ \diff u_n } ( \Omega )
        \leq
        \abs{ \diff u } ( \Omega \setminus \Omega_{ t_n } ) + C t_n \abs{ \diff_\tau (u \circ \Phi ) ( \cdot, t_n ) } ( \partial \Omega ),
    \end{equation*}
    which converges to $ \abs{ \diff u } ( \Omega ) $ by (\ref{eq:slice_bound}) and since $ t_n \to 0 $ as $ n \to \infty $. This finishes the proof.
\end{proof}

\begin{proof}[Proof of \Cref{p:first order gamma limit multidim}]
   First we show the $\liminf$-inequality. Let $ u \in \lp^1 ( \Omega ) $ and $ ( u_\eps )_\eps \subseteq \w^{1,2} ( \Omega ) $  such that $ u_\eps \to u $ in $ \mathrm{L}^1 ( \Omega ) $. By passing to a suitable non-relabelled subsequence, we may assume without loss of generality that $ \tr u_\eps = g_\eps $ on $  \partial \Omega $ and $ u_\eps \to u $ pointwise almost everywhere. Let $ \delta > 0 $ be as in \Cref{p:The Phi} and $ 0< \delta' < \delta $. Then
   \begin{align*}
       &\liminf_{\eps \to 0 }
       \int_\Omega 
       \frac{1}{\eps } W ( u_\eps ) + \eps \abs{\nabla u_\eps }^2 
       \dd{ x }
       \\
       \geq{}&
       \liminf_{\eps \to 0 }
       \int_{\Omega_\delta' }
       \frac{1}{\eps } W ( u_\eps ) + \eps \abs{\nabla u_\eps }^2 
       \dd{ x }
       +
       \liminf_{ \eps \to 0 }
       \int_{ \Omega \setminus \overline{\Omega_\delta' } }
       \frac{1}{\eps } W ( u_\eps ) + \eps \abs{\nabla u_\eps }^2 
       \dd{ x }.
   \end{align*}
   By the classical $ \Gamma $-convergence result \cite{ModicaMortola77}, the second summand is estimated from below by
   \begin{equation*}
       C_W \mathrm{Per} ( \{ u = a \} ; \Omega \setminus \overline{\Omega_\delta' } )
       \to 
       C_W \mathrm{Per} ( \{ u = a \} ; \Omega ) 
       \quad
       \text{as }
       \delta' \to 0.
   \end{equation*}
   For the first summand, we apply Fatou's Lemma and \Cref{p:The Phi} to obtain
   \begin{align*}
    \liminf_{\eps \to 0 }
       \int_{\Omega_\delta' }
       \frac{1}{\eps } W ( u_\eps ) + \eps \abs{\nabla u_\eps }^2 
       \dd{ x }   
    & \geq
    \int_{ \partial \Omega }
    \liminf_{ \eps \to 0 }
    \int_0^\delta 
    \left( \frac{1}{\eps} W ( u_\eps \circ \Phi ) + \eps \abs{ \nabla u_\eps \circ \Phi }^2 \right)
    \omega ( t, y ) 
    \dd{ t }
    \dd{ \hm^{n-1} ( y ) }.
   \end{align*}
    For given $ y \in \partial \Omega $, we define $ v_\eps^{(y)} ( t ) \coloneqq u_\eps ( \Phi (y,t ) )$. First we note that by passing to a non-relabelled subsequence, we have that
    \begin{equation*}
        v_\eps^{(y)}
        \to 
        u \circ \Phi ( y, \cdot ) 
        \text{ in } \mathrm{L}^1 ( (0,\delta ) )
        \text{ for } \hm^{n-1}\text{-almost every }y \in \partial \Omega.
    \end{equation*}
    By using a density argument and Fatou's Lemma, we see that in the sense of one-dimensional traces, we have $ v_\eps^{(y)}(0) = g_\eps ( y ) $ for $ \hm^{n-1}$-almost every $ y \in \partial \Omega $. 
    Moreover we know again by a density argument that for $ \lm^1 $-almost every $ 0< \delta' < \delta $, the interior trace of $ u $ respectively $ u_\eps $ taken with respect to $ \Omega_{\delta'} $ coincide with $ u $  respectively $ u_\eps $  $ \hm^{n-1}$-almost everywhere.
    By using a density argument again, we see that for $ \lm^1$-almost every  $ 0 < \delta' < \delta $, we have $ \tr v_\eps^{(y)} ( \delta' ) = u_\eps ( \Phi ( y, \delta' ) ) $  $ \hm^{n-1}$-almost everywhere and 
    $ u_\eps ( \Phi ( y, \delta') ) \to u  (\Phi (y, \delta') ) $ in $ \mathrm{L}^1 ( \partial \Omega_{\delta'} \cap \Omega ) $.
    Furthermore we note that $ \abs{\partial_t v_\eps^{(y)}} \leq \abs{ \nabla u_\eps \circ \Phi (y, t ) } $, thus by applying the one-dimensional $ \liminf$-inequality \Cref{p:liminf 1d}, we obtain for $ \lm^1$-almost every $ 0 < \delta' < \delta $ that
    \begin{align*}
    &\int_{ \partial \Omega }
    \liminf_{ \eps \to 0 }
    \int_0^\delta 
    \left( \frac{1}{\eps} W ( u_\eps \circ \Phi ) + \eps \abs{ \nabla u_\eps \circ \Phi }^2 \right)
    \omega ( t, y ) 
    \dd{ t }
    \dd{ \hm^{n-1} ( y ) }
    \\
    \geq{}&
    \int_{ \partial \Omega }
    \liminf_{ \eps \to 0 }
    \int_0^\delta 
    \left( \frac{1}{\eps} W ( v_\eps^{(y)} ) + \eps \abs{  \partial_t v_\eps^{(y)} }^2 \right)
    \omega ( t, y ) 
    \dd{ t }
    \dd{ \hm^{n-1} ( y ) }
    \\
    \geq{} &
    \int_{ \partial \Omega }
    C_W \mathrm{Per} ( \{ u (\Phi (y, \cdot ))= a \} , \Omega_{ \delta' }  )
    \\
    & \quad \;\;\,+
    \dw ( g ( y ) , \tr_{\partial \Omega } u ( y ) )
    +
    \dw ( u (\Phi ( y, \delta' ) ) , \tr_{\partial \Omega_{\delta'} \cap \Omega } u ( \Phi ( y, \delta' ) ) )
    \dd{ \hm^{n-1} ( y ) }.
    \end{align*}
    Here the last summand vanishes by our choice of $ \delta' $ above. Moreover, as $ \delta' $ tends to zero, the set $ \Omega_{ \delta' }$ vanishes, thus the first summand tends to zero. By combining this estimate with the above arguments, we obtain
    \begin{equation*}
        \liminf_{\eps \to 0 }
        \int_\Omega \frac{1}{\eps} W ( u_\eps ) + \eps \abs{ \nabla u_\eps }^2
        \dd{ x }
        \geq
        C_W \mathrm{Per} ( \{ u = a \} ; \Omega )
        +
        \int_{ \partial \Omega }
        \dw ( g ( y ) , \tr u ( y ) )
        \dd{ \hm^{n-1} ( y ) },
    \end{equation*}
    which is what we wanted to show.
   
    We now prove the $ \limsup $-inequality. First we note that it suffices to consider $ u \in \bv ( \Omega ; \{ a, b \} ) $ since in the other case, we may take any $ \lp^1 $-approximation through $ \w^{1,2}$-functions. Furthermore, we note that $ \mathcal{F}^{(1)} $ is continuous with respect to the strict convergence on $ \bv ( \Omega ) $ because the trace operator is continuous onto $ \lp^1 ( \partial \Omega ) $ with respect to the strict convergence on $ \bv ( \Omega ) $ and $ \mathrm{d}_W ( x, y ) \lesssim \abs{ x - y } $ for $ x, y \in [a,b ]$. Thus it suffices by \Cref{lem:approx_near_boundary} for the $ \limsup $-inequality to additionally assume that there exists some $ \delta \gtr 0 $ such that 
    \begin{equation}
    \label{eq:u_constant_close_to_boundary}
        u( \Phi ( t, y ) ) = \tr u ( y ) 
        \quad 
        \text{for } \hm^{N-1}\text{-a.e. }y \in \partial \Omega \text{ and all } 0 \less t \less \delta.
    \end{equation}
    First we construct the recovery sequence in the interior of $ \Omega $.
    Recall the definition (\ref{e:Modica-Mortola}) and (\ref{e:Gammaconv modica mortola without constraint}) of the Cahn--Hilliard functional $ E_\eps^{(1)} $ and the perimeter functional $ E^{(1)}$, both without boundary data.
    By the original $ \Gamma $-convergence result \cite[Thm.~1]{ModicaMortola77} we find for a given $ u \in \bv ( \Omega ; \{a,b\} ) $ a sequence $ (u_\eps)_\eps \subseteq \h^1 ( \Omega ) $ such that $ u_\eps \to u $ in $ \lp^1 ( \Omega )$ and
    \begin{equation*}
        \limsup_{\eps \to 0 }
        E_\eps^{(1)} ( u_\eps ) 
        \leq 
        C_W \mathrm{Per} ( \{u=a\} ; \Omega ).
    \end{equation*} 
    By approximation and a diagonal sequence argument, we can assume without loss of generality that $ u_\eps \in C^\infty ( \overline{\Omega } )$.
    We now argue that we can modify $ u_\eps $ up to a small error such that the resulting sequence has the same property (\ref{eq:u_constant_close_to_boundary}) near the boundary as $ u $. For this we consider $ \Omega_\delta \setminus \Omega_{\delta/2}$. Using that
    \begin{align*}
        \int_{\delta/2}^\delta
        \int_{\partial \Omega }
        \abs{u_\eps ( \Phi ( y, t ) ) - u ( \Phi ( y, t ) ) } 
        \dd{\hm^{ N - 1 } ( y ) }
        \dd{ t }
        & \leq
        C
        \int_{ \Omega }
        \abs{ u_\eps ( x ) - u( x ) }
        \dd{ x }
        \to 0 
        \shortintertext{and}
        \int_{\delta/2}^{\delta }
        \int_{ \partial \Omega }
        \frac{1}{\eps } W ( u_\eps ( \Phi ( y, t ) ) ) + \eps \abs{ \nabla u_\eps ( \Phi ( y, t ) ) }^2
        \dd{ \hm^{N-1} ( y ) }
        \dd{ t }
        & \leq 
        C E_\eps^{(1) } ( u_\eps )
        \leq 
        C,
    \end{align*}
    we find by an overlapping argument that there must exist a sequence $ t_\eps \in [ \delta/2, \delta ] $ such that
    \begin{align}
        \int_{\partial \Omega }
        \abs{u_\eps ( \Phi ( y, t_\eps ) ) - u ( \Phi ( y, t_\eps ) ) } 
        \dd{\hm^{ N - 1 } ( y ) }
        & 
        \to 0 
        \shortintertext{and}
        \label{eq:bound_on_slice_energy}
        \int_{ \partial \Omega }
        \frac{1}{\eps } W ( u_\eps ( \Phi ( y, t_\eps ) ) ) + \eps \abs{ \nabla u_\eps ( \Phi ( y, t_\eps ) ) }^2
        \dd{ \hm^{N-1} ( y ) }
        & \leq
        C_1.
    \end{align}
    We then define the function
    \begin{equation*}
        \tilde{u}_\eps ( x ) 
        \coloneqq
        \begin{cases}
            u_\eps ( x ), 
            &\text{if }x \in \Omega \setminus \Omega_{t_\eps }
            \\
            u_\eps ( \Phi ( y, t_\eps ) ), 
            & \text{if } x \in \Omega_{t_\eps}, x = \Phi (y, t ) 
        \end{cases}
    \end{equation*}
    Because $ u $ satisfies equation (\ref{eq:u_constant_close_to_boundary}), we obtain that 
    \begin{equation}
    \label{eq:convergence_of_straigt_seq}
    \tilde{u}_\eps \to u \text{ in }  \lp^1 ( \Omega )
    \text{ and }
    u_\eps ( \Phi ( \cdot, t_\eps ) ) 
    \to 
    \tr u \text{ in } \lp^1 ( \partial \Omega ).
    \end{equation}
     Moreover, we note that if $y \in \partial \Omega, t \in (0, t_\eps ) $, we have
    \begin{align*}
        \partial_t ( \tilde{u}_\eps \circ \Phi ) ( y, t )
        &= 
        0
        \shortintertext{and}
        \partial_{y_i}
        ( \tilde{u}_\eps \circ \Phi ) (y, t )
        & =
        \nabla u_\eps ( \Phi ( y, t_\eps ) ) \partial_{y_i} \Phi ( y, t_\eps ).
    \end{align*}
    Thus we obtain that
    \begin{equation*}
        \abs{ \nabla \tilde{u}_\eps ( \Phi ( y, t ) ) }
        =
        \abs{ \nabla ( \tilde{u}_\eps \circ \Phi ) (y, t ) \nabla \Phi^{-1} ( \Phi ( y, t ) ) }
        \lesssim 
        \abs{ \nabla u_\eps ( \Phi ( y, t_\eps ) ) }.
    \end{equation*}
    From this inequality and inequality (\ref{eq:bound_on_slice_energy}), we deduce the energy bound
    \begin{align}
        \notag
        &\limsup_{\eps \to 0 }
        \mathcal{F}_\eps^{(1)} ( \tilde{u}_\eps )
        \\
        \notag
        \leq{}&
        \limsup_{ \eps \to 0 }
        E_\eps^{(1)} ( u_\eps )
        +
        \int_{ \partial \Omega }
        \int_0^{t_\eps }
        \left( 
        \frac{1}{\eps} W ( u_\eps ( \Phi ( y, t_\eps ) ) ) + C \eps \abs{ \nabla u_\eps ( \Phi ( y, t_\eps ) ) }^2
        \right)
        \omega( y, t )
        \dd{ t }
        \dd{ \hm^{N- 1 } ( y ) }
        \\
        \label{eq:energy_of_straightened_sequence}
         \leq{}&
        \limsup_{ \eps \to 0}
        E_\eps^{(1)} ( u_\eps  )
        +
        C_1 \delta.
    \end{align}
    For ease of notation, we now replace $ u_\eps $ by $ \tilde{u}_\eps $.
    Recall the definition of $ \Psi_\alpha $ in (\ref{eq:def_Psialpha}). For $ y \in \partial \Omega $ define $ T_\eps ( y ) \coloneqq \eps \mathrm{sgn} ( \tr u_\eps ( y) - g_\eps ( y ) ) \Psi_{g_\eps ( y ) } ( \tr u_\eps ( y ) ) $ and let $ v_\eps ( y, \cdot ) \colon [ 0, T_\eps ( y ) ] \to [g_\eps ( y ) , \tr u_\eps ( y ) ] $ be the inverse of $ \eps \mathrm{sgn} ( \tr u_\eps ( y ) - g_\eps ( y ) ) \Psi_{g_\eps ( y ) }$. 
    For $ t \gtr T_\eps ( y ) $, we define $ v_\eps ( y , t ) \coloneqq u_\eps ( \Phi ( y, t ) ) $ so that $ v_\eps \colon \partial \Omega \times (0, \delta ) \to \R$. 
   Note that because $ u_\eps $ is constant along normal directions in $ \Omega_{\delta/2}$, we have continuity of $ v_\eps ( y, \cdot ) $ at time $ t = T_\eps ( y ) $ for $ \eps > 0 $ sufficiently small. 
   Moreover since $ g_\eps \in \h^1 ( \partial \Omega ) $, we see that $ v_\eps \in \h^1 ( \partial \Omega \times (0, \delta ) ) $. Furthermore if $ t \leq T_\eps (y ) $ its derivative is given by
   \begin{align}
   \label{eq:partial_t_v_eps}
       \partial_t v_\eps ( y, t ) 
       & =
       \mathrm{sgn} ( \tr ( u_\eps (y ) ) - g_\eps ( y ) )
       \frac{ \sqrt{W ( v_\eps ( y, t ) ) }}{ \eps }
       \shortintertext{and}
       \notag
       \nabla_\tau v_\eps ( y, t ) 
       & =
       \mathrm{sgn} ( \tr ( u_\eps (y ) ) - g_\eps ( y ) )
       \frac{ \sqrt{ W ( v_\eps ( y, t ) ) } }{ \sqrt{ W ( g_\eps ( y ) ) } }
       \nabla_\tau g_\eps ( y )
   \end{align}
   as in equation (\ref{eq:est_for_tang_der_veps}).
   Finally we define
   \begin{equation*}
       \bar{u}_\eps ( x )  
       \coloneqq 
       \begin{cases}
           v_\eps ( \Phi^{-1} ( x ) ), 
           & \text{if } x \in \Omega_\delta
           \\
           u_\eps ( x )
           & \text{else}.
       \end{cases}
   \end{equation*}
   By the chain rule, $ \bar{u}_\eps ( x ) \in \h^1 ( \Omega ) $ with $ \tr \bar{u}_\eps ( y ) = g_\eps ( y ) $
   and since $ v_\eps (y, t )$ is constant for $ t \geq T_\eps ( y )$, we have
   \begin{equation*}
       \abs{ \nabla \bar{u}_\eps ( x ) }^2
       \leq
       \begin{cases}
           \abs{ \nabla u_\eps ( x ) }^2,
           & \text{if }x\in \Omega \setminus \Omega_\delta \text{ or } x = \Phi (y, t ), t \gtr T_\eps ( y ) 
           \\
           \abs{ \partial_t v ( y,t  ) }^2 
           +
           C \abs{ \nabla_\tau v_\eps ( y ) }^2,
           & \text{if } x \in \Omega_\delta, x= \Phi ( y, t ), t \leq T_\eps ( y ).
       \end{cases}
   \end{equation*}
   We finally bound the energy of the sequence $ \bar{u}_\eps $ by
   \begin{align}
   \notag
       &\limsup_{\eps \to 0 }
       \mathcal{F}_\eps^{(1)} ( \bar{u}_\eps )
       \\
       \notag
       \leq{} &
       \limsup_{ \eps \to 0 } E_\eps^{(1)} ( u_\eps )
       +
       \limsup_{ \eps \to 0 }
       \int_{ \partial \Omega }
       \int_0^{T_\eps ( y ) }
       \left(
       \frac{ 1 }{\eps } W ( v_\eps ( y, t ) ) +
       \eps \abs{ \partial_t v_\eps ( y, t ) }^2 
       \right) \omega ( y, t ) 
       \dd{ t }
       \dd{ \hm^{N-1} ( y ) }
       \\
       \label{eq:limsup_partition}
       &+
       C
       \limsup_{\eps \to 0}
       \int_{ \partial \Omega }
       \int_0^{T_\eps}  \eps \abs{ \nabla_\tau v_\eps ( y ,t) }^2
       \dd{t}
       \dd{ \hm^{N-1} ( y ) }.
   \end{align}
   By inequality (\ref{eq:energy_of_straightened_sequence}), the first summand is bounded by
   \begin{equation}
   \label{eq:limsup_first_summand}
        \limsup_{\eps \to 0 }
        E_\eps^{(1)} ( u_\eps ) 
        \leq
       E^{(1)} ( u ) + C \delta.
   \end{equation}
   For the second summand, we compute by equation (\ref{eq:partial_t_v_eps}), the fact that $ \mathrm{d}_W ( x, y ) \leq C \abs{ x - y }$ for $ x, y \in [a,b]$, $ \dw ( g_\eps, g ) \to 0 $ in $ \lp^1 ( \partial \Omega ) $ and $ u_\eps ( \Phi ( \cdot, t ) ) \to \tr u $ in $ \lp^1 ( \partial \Omega ) $ by equation (\ref{eq:convergence_of_straigt_seq}) that
   \begin{align}
       \notag
       \int_{ \partial \Omega }
       \int_{0}^{T_\eps ( y ) }
       \frac{
       W ( v_\eps ( y, t ) ) }{\eps}
       +
       \eps \abs{ \partial_t v_\eps ( y, t ) }^2
       \dd{ \hm^{N-1} ( y ) }
       \dd{ t }
       & =
       \int_{ \partial \Omega }
       \int_0^{ T_\eps ( y ) }
       2 \sqrt{ W ( v_\eps ( y, t ) ) }
       \abs{ \partial_t v_\eps ( y, t ) }
       \dd{ t }
       \dd{ \hm^{N-1} ( y ) }
       \\
       \notag
       & =
       \int_{ \partial \Omega }
       \dw ( g_\eps ( y ), u_\eps ( \Phi ( y, t_\eps ) ) )
       \dd{ \hm^{N-1} } ( y )
       \\
       \label{eq:limsup_second_summand}
       & \to 
       \int_{ \partial \Omega }
       \dw ( g ( y ) , \tr u ( y ) )
       \dd{ \hm^{N-1} } (y)
   \end{align}
   The last summand (\ref{eq:limsup_second_summand}) is an error which vanishes in the limit $ \eps \to 0 $ since it is up to a constant bounded from above by
   \begin{equation}
   \label{eq:limsup_third_summand}
       C \limsup_{\eps \to 0}
       \eps^2
       \int_{ \partial \Omega }
       \frac{ \abs{ \nabla_\tau g_\eps ( y ) }^2 }{ W ( g_\eps ( y ) )} \dw ( g_\eps ( y ) , \tr u_\eps ( y ) ) 
       \dd{ \hm^{N-1} ( y ) },
   \end{equation}
   which vanishes by assumption (\ref{eq:ass_tang_deriv_squared}).
   Thus combining inequality (\ref{eq:limsup_partition}), (\ref{eq:limsup_first_summand}), equality (\ref{eq:limsup_second_summand}) and inequality (\ref{eq:limsup_third_summand}), we obtain 
   \begin{equation*}
       \limsup_{ \eps \to 0 }
       \mathcal{F}_\eps^{(1) } ( u_\eps )
       \leq
       \mathcal{F}^{(1)} ( u ) 
       +
       C \delta .
   \end{equation*}
   Since we could have chosen $ \delta $ to be arbitrarily small without enlarging the constant $ C $, it follows from the lower semicontinuity of the $ \Gamma $-$\limsup $ that the desired $ \limsup $-inequality must already hold, which finishes the proof.
\end{proof}

\section*{Acknowledgements and Funding}
The authors are very grateful for the hospitality of the Center for Nonlinear Analysis at Carnegie Mellon University, which hosted them during February and March of 2025. The project was recommended by Irene Fonseca and Giovanni Leoni there, with which the authors had many fruitful discussions.
P.S. was supported
by funding from the Deutsche Forschungsgemeinschaft (DFG, German Research Foundation) under Germany’s Excellence Strategy – EXC-2047/1 – 390685813, the DFG project 211504053 - SFB 106 and the Global Math Exchange Program, Hausdorff School of Mathematics.
F.C. has been supported by the European Union - Next Generation EU, Mission 4 Component 1 CUP B53D2300930006, codice 2022J4FYNJ, PRIN2022 project
“Variational methods for stationary and evolution problems with singularities and interfaces”.
F.C. is a member of GNAMPA - INdAM.

\addcontentsline{toc}{section}{References}

\bibliography{biblio.bib}
\bibliographystyle{plain}

\end{document}